\documentclass%
{amsart}
\usepackage{microtype, mathrsfs, xfrac, stmaryrd}
\usepackage{amssymb, enumerate, xspace}
\usepackage[latin1]{inputenc}
\usepackage{tikz-cd}

\usepackage{url}

\numberwithin{equation}{section}

\numberwithin{subsection}{section}

\let\originalleft\left
\let\originalright\right
\def\left#1{\mathopen{}\originalleft#1}
\def\right#1{\originalright#1\mathclose{}}

\newenvironment{enumeratea}
{\begin{enumerate}[\upshape (a)]}
{\end{enumerate}}

\newenvironment{enumerate1}
{\begin{enumerate}[\upshape (1)]}
{\end{enumerate}}

\newtheorem*{namedtheorem}{\theoremname}
\newcommand{\theoremname}{testing}

\newtheorem{theorem}{Theorem}[section]
\newtheorem{proposition}[theorem]{Proposition}
\newtheorem{proposition-definition}[theorem]
{Proposition-Definition}
\newtheorem{corollary}[theorem]{Corollary}
\newtheorem{lemma}[theorem]{Lemma}

\theoremstyle{definition}
\newtheorem{definition}[theorem]{Definition}

\newtheorem{conditions}[theorem]{Conditions}
\newtheorem{example}[theorem]{Example}
\newtheorem{examples}[theorem]{Examples}
\newtheorem{remark}[theorem]{Remark}

\theoremstyle{remark}

\newcommand\nome{testing}
\newcommand\call[1]{\label{#1}\renewcommand\nome{#1}}
\newcommand\itemref[1]{\item\label{\nome;#1}}
\newcommand\refall[2]{\ref{#1}~(\ref{#1;#2})}
\newcommand\refpart[2]{(\ref{#1;#2})}

\renewcommand{\mathcal}{\mathscr}

 \newcommand\cB{\mathcal{B}}
\newcommand\cC{\mathcal{C}} \newcommand\cD{\mathcal{D}}
 
\newcommand\cG{\mathcal{G}}

\newcommand\cO{\mathcal{O}}

 \newcommand\cV{\mathcal{V}}

\renewcommand\AA{\mathbb{A}} 
 
 \newcommand\FF{\mathbb{F}}
\newcommand\GG{\mathbb{G}}

 \newcommand\NN{\mathbb{N}}
 \newcommand\PP{\mathbb{P}}

 \newcommand\rD{\mathrm{D}}

 \newcommand\rR{\mathrm{R}}
 
\newcommand\rU{\mathrm{U}}

\newcommand\rma{\mathrm{a}}

\newcommand\rmm{\mathrm{m}} 
 \newcommand\rmp{\mathrm{p}}

\renewcommand{\epsilon}{\varepsilon}

\newcommand\arr{\ifinner\to\else\longrightarrow\fi}

\newcommand\arrto{\ifinner\mapsto\else\longmapsto\fi}

\newcommand{\xarr}{\xrightarrow}

\newcommand\noqed{\renewcommand\qed{}}

\renewcommand\H{\operatorname{H}}

\newcommand\op{^{\mathrm{op}}}

\newcommand\eqdef{\overset{\mathrm{\scriptscriptstyle def}} =}

\renewcommand{\th}{^\text{th}}

\def\displaytimes_#1{\mathrel{\mathop{\times}\limits_{#1}}}

\def\displayotimes_#1{\mathrel{\mathop{\bigotimes}\limits_{#1}}}

\renewcommand\hom{\operatorname{Hom}}

\newcommand\ext{\operatorname{Ext}}

\newcommand\aut{\operatorname{Aut}}

\newcommand\pic{\operatorname{Pic}}

\newcommand\spec{\operatorname{Spec}}

\newcommand\generate[1]{\langle #1 \rangle}

\newcommand\rk{\operatorname{rk}}

\newcommand\id{\mathrm{id}}

\newcommand\indlim{\varinjlim}

\newcommand\coker{\operatorname{coker}}

\renewcommand\projlim{\varprojlim}

\newcommand{\cat}[1]{(\mathrm{#1})}

\newcommand\double{\mathbin{\rightrightarrows}}

\newcommand\doublelong[2]{\mathbin{\xymatrix{{}\ar@<3pt>[r]^{#1}
\ar@<-3pt>[r]_{#2}&}}}

\newcommand{\underhom}{\mathop{\underline{\mathrm{Hom}}}\nolimits}

\newcommand{\underisom}{\mathop{\underline{\mathrm{Isom}}}\nolimits}

\newcommand{\underaut}{\mathop{\underline{\mathrm{Aut}}}\nolimits}

\newlength{\ignora}\newcommand{\hsmash}[1]{\settowidth{\ignora}{#1}#1\hspace{-\ignora}}

\renewcommand{\setminus}{\smallsetminus}

\renewcommand\projlim{\varprojlim}

\newcommand{\catset}{\cat{Set}}
\newcommand{\catab}{\cat{Ab}}

\newcommand{\mmu}{\boldsymbol{\mu}}

\newcommand{\gm}{\GG_{\rmm}}
\newcommand{\gmp}[1]{\GG_{\rmm, #1}}
\newcommand{\GL}{\mathrm{GL}}

\newcommand{\ga}{\GG_{\rma}}

\newcommand\radice[2]{\sqrt[\uproot{2}#1]{#2}}

\newcommand{\step}[1]{\smallskip\textit{#1.}\;}

\newcommand{\smallbullet}{\mathchoice{}{}%
{\raisebox{.14ex}{\scalebox{.6}\textbullet}}%
{\raisebox{.09ex}{\scalebox{.5}\textbullet}}%
}%

\newcommand{\aff}[1]{(\mathrm{Aff}/#1)}

\newcommand{\univ}[1]{\Pi^{\cC}_{#1/\kappa}}
\newcommand{\univprime}[1]{\Pi^{\cC}_{#1/\kappa'}}
\newcommand{\univd}[1]{\Pi^{\cD}_{#1/\kappa}}
\newcommand{\univmt}[1]{\Pi^{\mathrm{MT}}_{#1/\kappa}}

\newcommand{\univfin}[1]{\Pi^{\rm N}_{#1/\kappa}}
\newcommand{\univu}[1]{\Pi^{\mathrm{U}}_{#1/\kappa}}
\newcommand{\univvu}[1]{\Pi^{\mathrm{VU}}_{#1/\kappa}}
\newcommand{\univsvu}[1]{\Pi^{\mathrm{SVU}}_{#1/\kappa}}
\newcommand{\univa}[1]{\Pi^{\mathrm{A}}_{#1/\kappa}}
\newcommand{\univva}[1]{\Pi^{\mathrm{VA}}_{#1/\kappa}}
\newcommand{\univvn}[1]{\Pi^{\mathrm{VN}}_{#1/\kappa}}

\newcommand{\fund}{\pi_{1}^{\cC}}
\newcommand{\fundu}{\pi_{1}^{\rm U}}

\newcommand{\vu}{^{\rm VU}}

\newcommand{\svu}{^{\rm SVU}}

\newcommand{\red}{_{\mathrm{red}}}
\newcommand{\ored}{^{0}\red}
\newcommand{\git}{/\hspace{-2pt}/}

\newcommand{\rep}{\operatorname{Rep}}
\newcommand{\vect}{\operatorname{Vect}}

\newcommand{\cha}{\operatorname{char}}

\newcommand{\tors}{\operatorname{Tors}}

\newcommand{\sep}{^{\mathrm{sep}}}

\newcommand{\et}{_{\text{\rm\'et}}}

\newcommand{\thickslash}{\mathbin{\!\!\!\fatslash}}

\newcommand{\underpic}{\mathop{\underline{\mathrm{Pic}}}\nolimits}

\newcommand{\underger}{\mathop{\underline{\mathrm{Ger}}}\nolimits}

\newcommand{\gmet}[1]{\widetilde\cO_{#1\et}}

\newcommand{\frob}{\operatorname{Frob}}

\begin{document}

\title{Fundamental gerbes}

\author[Borne]{Niels Borne$^\dagger$}

\author[Vistoli]{Angelo Vistoli$\ddagger$}

\address[Borne]{Laboratoire Paul Painlevé\\
Université de Lille\\
U.M.R. CNRS 8524\\
U.F.R. de Mathématiques\\
59\,655 Villeneuve d'Ascq Cédex\\
France}
\email{Niels.Borne@math.univ-lille1.fr}

\address[Vistoli]{Scuola Normale Superiore\\Piazza dei Cavalieri 7\\
56126 Pisa\\ Italy}
\email{angelo.vistoli@sns.it}

\thanks{$^\dagger$Supported in part by the Labex CEMPI (ANR-11-LABX-0007-01) and Anr ARIVAF (ANR-10-JCJC 0107)}

\thanks{$^\ddagger$Supported in part by research funds from the Scuola Normale Superiore}

\begin{abstract}
For a class of affine algebraic groups $\cC$ over a field, we define the notions of $\cC$-fundamental gerbe of a fibered category, generalizing what we had done in \cite{borne-vistoli-fundamental-gerbe} for finite group schemes.

We give a necessary and sufficient conditions on $\cC$ implying that a fibered category $X$ over $\kappa$ satisfying mild hypotheses admits a Nori $\cC$-fundamental gerbe. We also give a tannakian interpretation of the gerbe that results by taking as $\cC$ the class of virtually unipotent group schemes, under a properness condition on $X$.
\end{abstract}

\maketitle


\section{Introduction}

\subsubsection*{Previous work} Let $X$ be a reduced proper connected scheme over a 
field $\kappa$, with a rational point $x_{0} \in X(\kappa)$. The celebrated result of Nori \cite{nori-phd} says the following.

\begin{enumerate1}

	\item There is a profinite group scheme $\pi(X, x_{0})$, the \emph{Nori fundamental group scheme}, with a $\pi(X, x_{0})$-torsor $P \arr X$ with a trivialization $P\mid_{x_{0}} \simeq \pi(X, x_{0})$ such that for every profinite group scheme $G \arr \spec \kappa$ and every $G$ torsor $Q \arr X$ with a trivialization $\alpha\colon Q \mid_{x_{0}} \simeq G$, there is a unique homomorphism of group schemes $\pi(X, x_{0}) \arr G$ inducing $Q$ and $\alpha$.

\item There is an equivalence of tannakian categories between representations of the group scheme $\pi(X, x_{0})$ and essentially finite locally free sheaves on $X$.

\end{enumerate1} 

In our paper \cite{borne-vistoli-fundamental-gerbe} we extend this result in three ways: 
\begin{enumerate1}

\item we relax greatly the hypotheses on $X$, 

\item we remove the dependence on the base point, which does not even need to exist, by replacing the fundamental group scheme $\pi(X, x_{0})$ with a \emph{fundamental gerbe $\univfin X$},

\item and we give a more general definition of essentially finite locally free sheaf on $X$.

\end{enumerate1}

The fundamental gerbe $\univfin X$ of a category $X$ fibered in groupoids over the category $\aff\kappa$ of affine schemes over a fixed base field $\kappa$ is a profinite gerbe with a morphism $X \arr \univfin X$, that is universal among morphisms from $X$ to a profinite gerbe.

Also in \cite{nori-phd}, Nori defines a \emph{unipotent fundamental group} scheme $\fundu(X,x_{0})$; it is a prounipotent group scheme with a $\fundu(X,x_{0})$-torsor $P \arr X$ that satisfies the analogue of the universal property above for torsors under prounipotent group schemes.

\subsubsection*{The motivating question} It is a natural question whether one can define a universal prounipotent gerbe $X \arr \univu X$.

More generally, suppose that we are given a class $\cC$ of affine algebraic groups of finite type defined over extensions of $\kappa$, satisfying some natural stability conditions, listed in Definition~\ref{def:stable}. Then one defines a \emph{$\cC$-gerbe over $\kappa$} as an affine fpqc gerbe $\Gamma \arr \aff \kappa$, such that for every extension $\ell$ of $\kappa$ and every object $\xi$ of $\Gamma(\ell)$, the group scheme $\underaut_{\ell}\xi$ of automorphisms of $\xi$ is in $\cC(\ell)$. A pro-$\cC$-gerbe is a gerbe that is a projective limit of $\cC$-gerbes. If $X$ is a fibered category, we define a $\cC$-fundamental gerbe as a pro-$\cC$ gerbe $\univ X$ with a morphism $X \arr \univ X$ which is universal among all maps from $X$ to a pro-$\cC$-gerbe. If $x_{0} \in X(\kappa)$, $\xi$ is the image of $x_{0}$ in $\univ X$, and $\underaut_\kappa\xi$ is the automorphism group scheme of $\xi$ over $\kappa$, then there exists an  $\underaut_\kappa\xi$-torsor $P \arr X$  satisfying the analogue of the universal property above for torsors under projective limits of group schemes in $\cC(\kappa)$. (See Section~\ref{sec:fundamental-gerbes} for the rigorous definitions.)

In this paper we answer the following question: under what conditions on $\cC$ does $\univ X$ exist for a reasonably large class of fibered categories? 

It is certainly not the case that it exists in general. For example, one can show that if $\cC$ contains the semidirect product $\gm \ltimes \ga$, then $\univ X$ does not exist every time $X$ is a scheme with a line bundle with a nonzero section that vanishes somewhere (Example~\ref{ex:does-not-exist}).

\subsubsection*{Existence results for fundamental gerbes} We characterize the classes $\cC$ for which $\univ X$ exists for reasonable general $X$. If $G$ is an affine group scheme of finite type over a field $k$, we say that $G$ is \emph{virtually nilpotent} if, after passing to the algebraic closure of $k$, the group $G$ contains a nilpotent subgroup scheme of finite index. \emph{Virtually unipotent} and \emph{virtually abelian} group schemes are defined similarly.

We say that a class $\cC$ is \emph{well-founded} when it consists of virtually nilpotent group schemes. Our main examples of well-founded classes are those of virtually abelian and virtually unipotent affine group schemes of finite type.

Our first main result, Theorem~\ref{thm:maintheorem}, states that that if $X$ satisfies a mild finiteness condition, it is geometrically reduced, in the sense of Definition~\ref{def:reduced}, and $\H^{0}(X, \cO) = \kappa$, then $\univ X$ exists for every well-founded class $\cC$. For schemes, the finiteness condition is equivalent to being quasi-compact and quasi-separated.

In fact, the condition that the class $\cC$ be well-founded turns out to be also necessary (Remark~\ref{rmk:necessary}). In other words, as soon as we admit a group in our class $\cC$ that is not virtually nilpotent, then fundamental gerbes $\univ X$ do not exist anymore for a wide class of quasi-projective schemes $X$ satisfying the conditions above.

The proof of Theorem~\ref{thm:maintheorem} is very similar in structure with that of the existence of the fundamental gerbe in \cite{borne-vistoli-fundamental-gerbe}.

There are many examples of well-founded classes, and, correspondingly, many fundamental gerbes, and fundamental group schemes, that one can associate with a fibered category as above. Here are some examples.

\begin{enumerate1}

\item The Nori fundamental gerbe $\univfin X$, associated with the class of finite group schemes.

\item The unipotent fundamental gerbe $\univu X$.

\item The virtually unipotent fundamental gerbe $\univvu X$.

\item The abelian fundamental gerbe $\univa X$.

\item The virtually abelian fundamental gerbe $\univva X$.

\item The fundamental gerbe of multiplicative type $\univmt X$, associated with the class of group schemes of multiplicative type.

\item The virtually nilpotent fundamental gerbe $\univvn X$. Since, by definition, a well-founded class is contained in the class of virtually nilpotent groups, and fundamental gerbes are functorial under inclusion of classes (see Section~\ref{sec:change-class}), the virtually unipotent group fundamental gerbe $\univvn X$ dominates all the other $\univ X$ (we can call it the One Gerbe, in analogy with Tolkien's One Ring).

\end{enumerate1}

\subsubsection*{The tannakian interpretations} Of course one would like to have a tannakian interpretation for each of the fundamental gerbes above.

If $\Gamma$ is an affine gerbe over $\kappa$, we denote by $\rep \Gamma$ the corresponding tannakian category. A morphism $X \arr \Gamma$ induces a pullback $\rep\Gamma \arr \vect_{X}$, where we denote by $\vect_{X}$ the category of locally free sheaves on $X$; in particular for every fibered category satisfying the conditions of Theorem~\ref{thm:maintheorem} and every well-founded class $\cC$, we obtain a functor $\rep\univ X \arr \vect_{X}$. The pullbacks $\rep\univvn X \arr \vect_{X}$, $\rep\univa X \arr \vect_{X}$ and $\rep\univva X\arr \vect_{X}$ are almost never fully faithful, and we are not able to give a non-tautological tannakian interpretation of $\rep\univa X$ and $\rep\univva X$.

In contrast with this, we have that if $\cC$ is a well-founded subclass of the class of virtually unipotent group schemes, the pullback $\rep\univ X \arr \vect_{X}$ is fully faithful (Corollary~\ref{cor:fully-faithful-2}). In particular, the pullbacks $\rep\univfin X \arr \vect_{X}$, $\rep\univu X \arr \vect_{X}$ and $\rep\univvu X \arr \vect_{X}$ are fully faithful.

The pullback $\rep\univfin X \arr \vect_{X}$ induces an equivalence between $\rep\univfin X$ and the category of essentially finite bundles on $X$: this is proved in \cite{borne-vistoli-fundamental-gerbe}.

The pullback $\rep\univu X \arr \vect_{X}$ induces an equivalence of $\rep\univu X$ with the class of locally free sheaves that are obtained from successive extensions from trivial bundles (Theorem~\refall{thm:tannakian-char-u-vu}{1}). This is generalization of the tannakian characterization of the unipotent fundamental group scheme due to Nori \cite{nori-phd}, and is not at all surprising.

The tannakian interpretation of the virtually unipotent gerbe $\univvu X$ is somewhat more interesting. In \cite{otabe-extension}, S{.} Otabe defined \emph{semi-finite bundles}: these are locally free sheaves that are obtained as successive extensions of essentially finite bundles; see Definition~\ref{def:vect-u-vu}. Our terminology is different, as we call these \emph{extended essentially finite locally free sheaves}.

If $\cha\kappa = 0$ we show that the pullback $\rep\univvu X \arr \vect_{X}$ gives an equivalence between $\rep\univvu X$ and the category of extended essentially finite  locally free sheaves on $X$ (Theorem~\refall{thm:tannakian-char-u-vu}{2}). If $\cha\kappa > 0$, then $\rep\univvu X$ is equivalent to the category of locally free sheaves that become extended essential finite bundles after pullback by a sufficiently high power of the absolute Frobenius (Theorem~\ref{thm:tannakian-characterization-3}).

We conclude with a reference to a result, Theorem~\ref{thm:tonini-zhang}, due to Tonini and Zhang. Assume that $\cha \kappa > 0$, that $X$ is a pseudo-proper geometrically reduced algebraic stack of finite type over $\kappa$, and that\/ $\H^{1}(X, E)$ is a finite-dimensional vector space over $\kappa$ for all locally free sheaves on $X$. Then $\univvu X = \univsvu X = \univfin X$.

\subsubsection*{Fundamental gerbes of multiplicative type} A particularly interesting fundamental gerbe is the fundamental gerbe $\univmt X$ of multiplicative type, as it gives a conceptual interpretation of the universal torsor of Colliot-Thélène and Sansuc \cite{colliot-thelene-sansuc-torsors, colliot-thelene-sansuc-descente-II}. Suppose that $X$ is a projective variety over a field $\kappa$, such that if $\kappa\sep$ is a separable closure of $\kappa$, then $\pic(X_{\kappa\sep})$ is a free abelian group of finite rank, and $x_{0} \in X(\kappa)$. Then Colliot-Thélène and Sansuc define a $G$-torsor on $X$, where $G$ is the torus associated with the action of the Galois group on $\pic(X_{\kappa\sep})$. In our language, $G$ is the fundamental group of multiplicative type of $(X, x_{0})$.

In the last section we give a direct construction of $\univmt X$, generalizing that of Colliot-Thélène and Sansuc, which is completely independendent of the general machinery in the rest of the paper; this works even for fibered categories $X$ satisfying the same mild finiteness condition, with $\H^{0}(X, \cO) = \kappa$, but without assuming that $X$ is geometrically reduced. We construct $\univmt X$ from the Picard stack $\underpic_{X}$ over the small étale site $\kappa\et$, which sends every étale $\kappa$-algebra $A$ into the groupoid $\underpic(X_{A})$ of invertible sheaves over $X_{A}$ (Theorem~\ref{thm:main-multiplicative}).

Along the way, we prove a very general duality theorem for gerbes of multiplicative type, which extends the well known duality between groups of multiplicative type and sheaves of abelian groups on $\kappa\et$ (or, equivalently, abelian groups with a continous action of the Galois group of $\kappa\sep/\kappa$). More precisely, we prove an equivalence of $2$-categories between gerbes of multiplicative type and a certain class of Picard stacks with additional structure (Theorem~\ref{thm:main-equivalence}). See also \cite{brochard-duality} and \cite[\S 2.4]{braverman-bezrukavnikov} for related ideas. 

(In this part we do not give the full details of all the proofs, as said details tend to be rather tedious.)

\subsection*{Description of content}

The first three sections of the paper aim at fixing the notation, and present some facts about affine gerbes and fibered categories which are undoubtedly known to the experts, but for which we could not find a suitable treatment in the literature.

The real action starts in Section~\ref{sec:fundamental-gerbes}, in which we give the general definition of a $\cC$-fundamental gerbe, explain the connection of this with the notion of $\cC$-fundamental group, and give examples to show how fundamental gerbes don't exist in general.

Section~\ref{sec:well-founded} contains the definition of a well-founded class, and several technical results on group scheme actions on affine varieties that lead to the characterization of well-founded classes given in Theorem~\ref{thm:char-well-founded}.

The first main result, the existence of $\univ X$ for a well-founded class $\cC$, with appropriate hypotheses on $X$, is in Section~\ref{sec:maintheorem}.

Section~\ref{sec:change-class} contains a small but very useful result on the relation between $\univ X$ and $\univd X$, when $\cD$ is a subclass of a well-founded class $\cC$.

Section~\ref{sec:base-change} contains a base-change result for $\univ X$ under an algebraic extension of $\kappa$.

Our main results on the tannakian interpretation of certain fundamental gerbes, Theorems \ref{thm:tannakian-char-u-vu} and \ref{thm:tannakian-characterization-3} are stated in Section~\ref{sec:tannakian}, together with the result of Tonini and Zhang, Theorem~\ref{thm:tonini-zhang}, mentioned above. In Section~\ref{sec:general-tannakian} we put the problem of giving a tannakian interpretation of fundamental gerbes for a certain fundamental class into a more general framework, and we prove a more general result (Theorem~\ref{thm:tannakian-characterization-1}) that implies \ref{thm:tannakian-char-u-vu}. Section~\ref{sec:proof-3} contains the proof of \ref{thm:tannakian-characterization-3}.

The last section contains our treatment of the duality theorem for gerbes of multiplicative type, and our alternative construction for the universal gerbe of multiplicative type.

\subsection*{Acknowledgments}

We are grateful to Sylvain Brochard and Lei Zhang for very useful discussions. We heartily thank Marta Pieropan, who pointed out to us the possible connection of our theory of fundamental gerbes with the theory of the universal torsor of Colliot-Thélène and Sansuc.

We are especially in debt with Fabio Tonini for several helpful remarks.

The fact that Theorem~\ref{thm:char-well-founded} should be true was pointed out to us by Andrei Okounkov and Johan De Jong, to whom we express our appreciation.

\section{Notations and conventions}

We will fix a base field $\kappa$. All schemes and morphisms will be defined over $\kappa$. All fibered categories will be fibered in groupoids over the category $\aff \kappa$ of affine $\kappa$-schemes (or, equivalently, over the opposite of the category of $\kappa$-algebras). A base-preserving functor between categories fibered in groupoids will be referred to in short as a map, or a morphism. A $\kappa$-scheme $U$ will be identified with the category fibered in sets $\aff{U} \arr \aff{\kappa}$, where $\aff U$ is the category of maps $T \arr U$, where $T$ is an affine scheme.

All group schemes will be affine over extensions $\ell$ of $\kappa$. If $G$ is a group scheme of finite type over $\ell$, we will denote by $G^{0}$ the connected component of the identity. If $\ell$ is perfect, $G\ored$ is a smooth connected subgroup scheme of $G$.

If a group scheme $G$ over an extension $\ell$ of $\kappa$ acts on an $\ell$-scheme $X = \spec A$, we denote by $X\git G$ the spectrum of the $\ell$-algebra of invariants $A^{G}$. We will need the following standard fact, which is, for example, a particular case of Grothendieck's result on the existence of quotients for finite flat groupoids (see \cite{grothendieck-quotients}).

\begin{lemma}\label{lem:finite-actions}
Assume that $\ell$ is algebraically closed and $G$ is finite over $\ell$. Then the fibers of the function $X(\ell) \arr (X\git G)(\ell)$ are precisely the orbits of the action of $G(\ell)$ on $X(\ell)$.
\end{lemma}

Let $G \arr \spec \ell$ be a group scheme, $P \arr \spec\ell$ a $G$-torsor. We can use the conjugation action of $G$ and $P$ to define a twisted form of $G$, which we call, as usual, an \emph{inner form} of $G$.

\section{Generalities on affine gerbes}

By \emph{affine gerbe} we will always mean affine fpqc gerbe over the base field $\kappa$, that is an fpqc gerbe over $\aff{\kappa}$ with affine diagonal, possessing an affine chart. These admit an obvious description in terms of groupoids (see \cite[\S 3]{borne-vistoli-fundamental-gerbe}) and are called \emph{tannakian gerbes} in \cite[Chapitre~III, \S 2]{saavedra}.

We will often consider gerbes of finite type, that by definition are those satisfying the equivalent conditions of the following proposition.

\begin{proposition}\call{finite-type}\label{prop:gerbe-finite-type}
Let $\Gamma \arr \spec \kappa$ be an affine gerbe. Then the following are equivalent.

\begin{enumerate1}

\itemref{8} $\Gamma$ is a smooth algebraic stack over $\kappa$.

\itemref{7} $\Gamma$ is an algebraic stack of finite type over $\kappa$.

\itemref{1} $\Gamma$ is an algebraic stack.

\itemref{2} The diagonal of\/ $\Gamma$ is of finite type.

\itemref{3} If $\ell$ is an extension of $\kappa$ and $\xi\in \Gamma(\ell)$, then $\underaut_{\ell}\xi$ is of finite type over $\ell$.

\itemref{4} There exists an extension $\ell$ of $\kappa$ and an object $\xi\in \Gamma(\ell)$ such that $\underaut_{\ell}\xi$ is of finite type over $\ell$.

\itemref{5} If $\{A_{i}\}_{i \in I}$ is an inductive system of $\kappa$-algebras, the natural map
   \[
   \indlim_{i} \Gamma(A_{i}) \arr \Gamma(\indlim_{i}A_{i})
   \]
is an equivalence of categories.

\itemref{6} The tannakian category $\rep\Gamma$ is finitely generated.
\end{enumerate1}

\end{proposition}

\begin{proof}
The implications \refpart{finite-type}{8}$\implies$\refpart{finite-type}{7}$\implies$\refpart{finite-type}{1}$\implies$%
\refpart{finite-type}{2}$\implies$\refpart{finite-type}{3}$\implies$\refpart{finite-type}{4} are obvious.

The proof of \refpart{finite-type}{1}$\implies$\refpart{finite-type}{8} is given for fppf gerbes  in \cite[Proposition~A.2]{bergh-destackification}, it is also valid for fpqc gerbes.

Here is a sketch of proof that \refpart{finite-type}{4} implies \refpart{finite-type}{5}. Set $G \eqdef \underaut_{\ell}\xi$; then $\Gamma_{\ell} = \cB_{\ell}G$ is an algebraic stack, as it follows from Artin's theorem (\cite[Théorème~10.1]{laumon-moret-bailly}). Since \refpart{finite-type}{1}$\implies$\refpart{finite-type}{7} holds,  $\Gamma_{\ell}$ is an algebraic stack of finite type over $\ell$, or equivalently, of finite presentation. Hence $\Gamma_{\ell}$  preserves filtered colimits (see 
\cite[{Tag 0123}]{stacks-project}). Set $R_{0} \eqdef \ell$, $R_{1} \eqdef \ell\otimes_{\kappa}\ell$, and $R_{2} \eqdef \ell\otimes_{\kappa}\ell\otimes_{\kappa}\ell$; the natural maps $\iota_{1}$, $\iota_{2}\colon R_{0} \arr R_{1}$ and $\iota_{12}$, $\iota_{13}$ and $\iota_{23}\colon R_{1} \arr R_{2}$ induce functors $\iota_{1*}$, $\iota_{2*}\colon \Gamma(A\otimes_{\kappa}R_{0}) \arr \Gamma(A\otimes_{\kappa}R_{1})$ and $\iota_{12*}$, $\iota_{13*}$ and $\iota_{23*}\colon \Gamma(A\otimes_{\kappa}R_{1}) \arr \Gamma(A\otimes_{\kappa}R_{2})$ for each $\kappa$-algebra $A$. We call $\Delta\arr \aff \kappa$ the fibered category of objects of $\Gamma$ with descent data along the covering $\spec \ell \arr \spec \kappa$; if $A$ is a $\kappa$-algebra, the objects of $\Delta(A)$ are pairs $(\xi, a)$, where $\xi$ is an object of $\Gamma(A\otimes_{\kappa}R_{0})$ and $a$ is an isomorphism $a\colon \iota_{2*}\xi \simeq \iota_{1*}\xi$ satisfying $\iota_{12*}a \circ \iota_{23*}a = \iota_{13*}a$. An arrow $f\colon (\xi, a) \arr (\eta,b)$ in $\Delta(A)$ is a arrow $f\colon \xi \arr \eta$ in $\Gamma(R\otimes A)$, with the property that $b \circ\iota_{2*}f = \iota_{1*}f \circ a$ in $\hom_{\Gamma(R_{1}\otimes A)}(\iota_{2*}\xi, \iota_{1*}\eta)$. So the diagram
   \[
   \begin{tikzcd}[column sep = 1em]
   \hom_{\Delta(A)}\bigl((\xi, a), (\eta, b)\bigr) \rar& \hom_{\Gamma(R_{0}\otimes A)}(\xi, \eta)
   \ar[rr, shift left, "f {\mapsto} b \circ\iota_{2*}f"]
   \ar[rr, shift right, swap, "f {\mapsto} \iota_{1*}f \circ a"]&\ &
   \hom_{\Gamma(R_{1}\otimes A)}(\iota_{2*}\xi, \iota_{1*}\eta)
   \end{tikzcd}
   \]
is an equalizer.

The obvious functor $\Gamma(A) \arr \Delta(A)$ is an equivalence, because $\Delta$ is an fpqc stack. Hence it is enough to prove that for any inductive system of $\kappa$-algebras $\{A_{i}\}$ the functor $\indlim_{i} \Delta(A_{i}) \arr \Delta(\indlim_{i}A_{i})$ is an equivalence. 

Let us show that $\indlim_{i} \Delta(A_{i}) \arr \Delta(\indlim_{i}A_{i})$ is fully faithful. For this, notice that if $R$ is an $\ell$-algebra, and $A$ a $\kappa$-algebra, then the fibered category sending $A$ into $\Gamma(R\otimes_{\kappa}A) = \Gamma_{\ell}\bigl(R\otimes_{\ell}(\ell\otimes_{\kappa}A)\bigr)$ preserves filtered colimits, because $\Gamma_{\ell}$ does, and tensor products preserve colimits.

Set $A \eqdef \indlim_{i}A_{i}$; we need to prove that the functor $\indlim_{i} \Delta(A_{i}) \arr \Delta(\indlim_{i}A_{i})$ is fully faithful. Take two objects $\{(\xi_{i}, a_{i})\}$ and $\{(\eta_{i}, b_{i})\}$ of $\indlim_{i}\Delta(A_{i})$; call $(\xi, a)$ and $(\eta, b)$ their images in $\Delta(A)$. By definition we have
   \[
   \hom_{\indlim \Delta(A_{i})}\bigl(\{(\xi_{i}, a_{i})\}, \{(\eta_{i}, b_{i})\}\bigr)
   = \indlim_{i}\hom_{\Delta(A_{i})}\bigl((\xi_{i}, a_{i}), (\eta_{i}, b_{i})\bigr)\,.
   \]
Since filtered colimits preserve equalizers, we have a commutative diagram
   \[
   \begin{tikzcd}[column sep = 1.5em]
   \indlim\hom\bigl((\xi_{i}, a_{i}), (\eta_{i}, b_{i})\bigr) \rar\dar& 
   \indlim\hom(\xi_{i}, \eta_{i})\dar
   \ar[r, shift left]
   \ar[r, shift right]&
   \indlim\hom(\iota_{2*}\xi_{i}, \iota_{1*}\eta_{i})\dar\\   
   \hom\bigl((\xi, a), (\eta, b)\bigr) \rar& \hom(\xi, \eta)
   \ar[r, shift left]
   \ar[r, shift right]&
   \hom(\iota_{2*}\xi, \iota_{1*}\eta)
   \end{tikzcd}
   \]
in which the rows are equalizers, and the last two columns are bijections. It follows that the first column is also a bijection, which is exactly what we want to show.

The proof of the fact that the functor $\indlim_{i} \Delta(A_{i}) \arr \Delta(\indlim_{i}A_{i})$ essentially surjective is easy, and left to the reader.

It is easy to check that \refpart{finite-type}{5} implies \refpart{finite-type}{2}: this follows from the well-known fact, due to Grothendieck, that an affine scheme over a ring $R$ is finitely presented if and only the functor on $R$-algebras that it represents preserves inductive limits.

Let us check the stronger result that \refpart{finite-type}{5} implies 
\refpart{finite-type}{1}. Let $\ell$ be an extension of $\kappa$ such that $\Gamma(\ell) \neq \emptyset$. If $\{A_{i}\}$ is the inductive system of $\kappa$-subalgebras of $\ell$ of finite type over $\kappa$; then $\indlim_{i}A_{i} = \ell$, so $\indlim_{i}\Gamma(A_{i}) \simeq \Gamma(\indlim_{i}A_{i}) \neq \emptyset$. Hence $\Gamma(A_{i}) \neq \emptyset$ for some $i$; by passing to a quotient by a maximal ideal of $A_{i}$ we see that there is a finite extension $k/\kappa$ such that $\Gamma(k) \neq \emptyset$. If $\xi \in \Gamma(k)$ and $G \eqdef \underaut_{k}\xi$, then 
$G$ is of finite type over $k$, because \refpart{finite-type}{2} is satisfied; so the map $\spec k \arr \Gamma_{k}$ is an fppf cover, hence the composite $\spec k \arr \Gamma_{k} \arr \Gamma$ is an fppf cover. From Artin's theorem (\cite[Théorème~10.1]{laumon-moret-bailly}) we see that $\Gamma$ is an algebraic stack, as claimed.

The equivalence between \refpart{finite-type}{6} and \refpart{finite-type}{2} is proven in \cite[III~3.3.1.1]{saavedra}.
\end{proof}

We will mainly need the following Corollary.

\begin{corollary}\label{lem:finite-separable-point}
Let $\Gamma$ be an affine gerbe of finite type over $\kappa$. Then there exists a finite separable extension $\kappa'/\kappa$ such that $\Gamma(\kappa') \neq \emptyset$.
\end{corollary}

\begin{proof}
	This follows from Proposition \ref{prop:gerbe-finite-type}\refpart{finite-type}{8}. 
\end{proof}

In this paper we will use repeatedly the following Lemma.

\begin{lemma}\label{lem:fiber-product}
Let $\phi'\colon G'\arr G$ and $\phi''\colon G'' \arr G$ be homomorphism of algebraic groups over $\kappa$. Then have an equivalence of fibered categories
   \[
   \cB_{\kappa}G' \times_{\cB_{\kappa}G} \cB_{\kappa}G'' \simeq 
   [G/(G' \times G'')]
   \]
where the action of $G' \times G''$ on $G$ is defined by
   \[
   g\cdot (g', g'') \eqdef \phi'(g')^{-1} g \phi''(g'')\,.
   \]
\end{lemma}

\begin{proof}
	An object of the fibered product $\cB_{\kappa}G' \times_{\cB_{\kappa}G} \cB_{\kappa}G''$ over a scheme $T$ is a triple $(P, P'', \alpha)$, where $P' \arr T$ and $P'' \arr T$ are, respectively, a $G'$-torsor and a $G''$-torsor, and $\alpha\colon P'\times^{G'}G \simeq P''\times^{G''}G$ is an isomorphism of $G$-torsors. Set $Q \eqdef P'\times^{G'}G$, and consider the usual isomorphism $\rho\colon Q\times G \arr Q\times_{T} Q$ defined by $(q, g) \arr (q, qg)$; denote by $\pi\colon Q\times_{T} Q \arr G$ the composite of $\rho^{-1}$ with the projection $Q\times G \arr G$. Let $u'\colon P'\arr Q$ the usual $\phi'$-equivariant morphism, and call $u''\colon P'' \arr Q$ the composite of the $\phi''$-equivariant morphism 
	$P'' \arr P''\times^{G''}G$ with $\alpha^{-1}$.

The composite
   \[
   P'\times_{T}P'' \xarr{u' \times u''} Q \times_{T} Q \xarr{\pi} G
   \]
is easily seen to be $(G'\times G'')$-equivariant, when the action of $G' \times G''$ on $G$ is the one described above. Of course $P'\times_{T}P''$ is a $(G' \times G'')$-torsor.

This defines a base-preserving function from the objects of $\cB_{\kappa}G' \times_{\cB_{\kappa}G} \cB_{\kappa}G''$ to those of $[G/(G' \times G'')]$; this is immediately seen to extend to a base-preserving functor $\cB_{\kappa}G' \times_{\cB_{\kappa}G} \cB_{\kappa}G'' \arr[G/(G' \times G'')]$.

To go in the opposite direction, let $P \arr T$ be a $(G'\times G'')$-torsor and $\theta\colon P \arr G$ a $(G'\times G'')$-equivariant morphism. If $P' \arr T$ and $P''\arr T$ are $G'$ and $G''$-torsors associated with $P$, we have a canonical isomorphism $P \simeq P'\times_{T}P''$; so we get a $(G'\times G'')$-equivariant morphism  $\theta\colon P'\times_{T}P'' \arr G$. From a section $p' \in P'(T)$ we obtain a $\phi''$-equivariant morphism $\theta_{p'}\colon P'' \arr G$, which in turn yields a $G$-equivariant morphism $P''\times^{G''}G \arr G$, which gives a section of $(P''\times^{G''}G)(T)$. Sending $p'$ into $\theta_{p'}$ gives a $\phi'$-equivariant morphism $P' \arr P''\times^{G''}G$, which extends to an isomorphism of $G$-torsors $P'\times^{G'}G \simeq P''\times^{G''}G$. This yields a base-preserving functor $[G/(G' \times G'')] \arr \cB_{\kappa}G' \times_{\cB_{\kappa}G} \cB_{\kappa}G''$, which is a quasi-inverse to the one above.
\end{proof}

Let $f\colon \Gamma \arr \Delta$ be a morphism of affine gerbes. Then $f$ is faithful if and only if for some extension $\ell$ of $\kappa$, and some object $\ell$ of $\Delta(\kappa)$, the induced homomorphism of group scheme $\underaut_{\ell}\xi \arr \underaut_{\ell}f(\xi)$ is a monomorphism. Hence, a homomorphism of group schemes $G \arr H$ induces a faithful morphism $\cB_{\kappa}G \arr \cB_{\kappa}H$ if and only if $G \arr H$ is a monomorphism.

If $\Gamma$ and $\Delta$ are of finite type, then $f$ is faithful if and only if it is representable. 

\begin{definition}
Let $f\colon \Gamma \arr \Delta$ be a morphism of affine gerbes over $\kappa$. We say that $f$ is \emph{locally full} if for any extension $\ell$ of $\kappa$ and any object $\xi$ of $\Gamma(\ell)$, the induced homomorphism of group schemes $\underaut_{\ell}\xi \arr \underaut_{\ell}f(\xi)$ is faithfully flat.
\end{definition}

\begin{remark}
If this is true for an extension $\ell$ and an object $\xi$ of $\Gamma(\ell)$, then it is true for all $\ell$ and all $\xi$.
\end{remark}

\begin{remark}
If $\phi\colon G \arr H$ is a homomorphism of affine group schemes, the corresponding morphism $\cB_{\kappa}G \arr \cB_{\kappa}H$ is locally full if and only if $\phi$ is 
faithfully flat, or, equivalently, an fpqc cover.
\end{remark}

\begin{remark}
A morphism of affine gerbes that is both faithful and locally full is in fact an equivalence. 
\end{remark}

\begin{definition}
Let $f\colon \Gamma \arr \Delta$ be a morphism of affine gerbes. A \emph{canonical factorization} of $f$ consists of a factorization $\Gamma \arr \Delta' \arr \Delta$ of $f$, such that $\Gamma \arr \Delta'$ is locally full, and $\Delta' \arr \Delta$ is faithful.
\end{definition}

\begin{proposition}
A morphism of affine gerbes $f\colon \Gamma \arr \Delta$ has a canonical factorization. Furthermore, if $\Gamma \arr \Delta' \arr \Delta$ and $\Gamma \arr \Delta'' \arr \Delta$ are two canonical factorizations, there exists an equivalence $\phi\colon \Delta' \arr \Delta''$, and a commutative diagram
   \[
   \begin{tikzcd}
   &\Delta' \arrow[dd, dashed, "\phi", near start] \ar{dr}&\\
   \Gamma\ar{ur}\ar{dr} \ar[rr, crossing over, "f", near start]&&\Delta\\
   & \Delta''\ar{ur} &\hsmash{\,.}\\
   \end{tikzcd}
   \]
\end{proposition}

\begin{proof}[Sketch of proof] For a tannakian proof, see \cite[Proposition B.4]{tonini-zhang-fundamental}; here is a direct approach.  Suppose that we are given canonical factorization $\Gamma \xarr{g} \Delta' \xarr{h} \Delta$. For each $\kappa$-algebra and each $\xi \in \Gamma(A)$, call $K_{A}(\xi)$ the kernel of the homomorphism of group schemes $\underaut_{A}(\xi) \arr \underaut_{A}\bigl(f(\xi)\bigr)$. If $\xi$, $\eta \in \Gamma(A)$, then $\underhom_{A}(\xi, \eta)$ is a torsor over $\spec A$ for the group $\underaut_{A}(\xi)$, thus, by restriction we obtain a free action of $K_{A}(\xi)$ on $\underhom_{A}(\xi, \eta)$. From the definition of $\Delta'$ it follows that the map $\underhom_{A}(\xi, \eta) \arr \underhom_{A}\bigl(f(\xi), f(\eta)\bigr)$ induces an isomorphism of fpqc sheaves
   \[
   \underhom_{A}(\xi, \eta)/K_{A}(\xi) \simeq \underhom_{A}\bigl(f(\xi), f(\eta)\bigr)
   \]
Hence $\Delta'$ is the fqpc stackification of the prestack whose objects are the objects of $\Gamma$, and whose arrows $\xi \arr \eta$ over a fixed $A$ are sections over $\spec A$ of the sheaf of sets $\underhom_{A}(\xi, \eta)/K_{A}(\xi)$. In other words, $\Delta'$ is the rigidification of $\Gamma$ along $K_{A}$, as defined in \cite[Appendix~A]{dan-olsson-vistoli1}. This shows the uniqueness of $\Delta'$.

For the existence, we prefer not to use the dubious notion of fqpc stackification, and do the following. Let $U \arr \Gamma$ be an fqpc cover by an affine scheme (for this is it enough that $U$ is a non-empty scheme, for example, the spectrum of a field). Then if we set $R \eqdef U \times_{\Gamma} U$ we get an fpqc groupoid $R \double U$. Set $U\times U = \spec A$, and call $\xi$ and $\eta$ the objects of $\Gamma(A)$ corresponding to the composites of the given morphism $U \arr \Gamma$ with the two projections $U \times U \arr U$; then $R$ represents the functor $\underhom_{A}(\xi, \eta)$. Hence there is an action of $K_{A}(\xi)$ on $R$, leaving the morphism $R \arr U \times U$ invariant. By passing to the fpqc quotient $R/K_{A}(\xi)$ we obtain a groupoid $R/K_{A}(\xi) \double U$, whose stack of torsors is the desired rigidification. There remains to prove that $R/K_{A}(\xi)$ is an affine scheme, as this shows that the rigidification is an affine gerbe, and ends the proof.

For this, let $K$ be a field extension of $\kappa$ and let $\zeta$ be an object of $\Gamma(K)$. Since $\Gamma$ is an fpqc gerbe, there exists a faithfully flat extension $A \subseteq B$ such that $K \subseteq B$, such that $\xi_{B}$, $\eta_{B}$ and $\zeta_{B}$ are all isomorphic. So the pullback of $R$ to $\spec B$ is isomorphic to the affine group scheme $\underaut_{B}(\zeta_{B})$, and $\bigl(R/K_{A}(\xi)\bigr)_{B}$ is the quotient $\underaut_{B}(\zeta_{B})/K_{B}(\zeta_{B})$, which is affine. Hence $R/K_{A}(\xi)$ is affine.
\end{proof}

\begin{proposition}\call{prop:char-locally-full}
Let $f\colon \Gamma \arr \Delta$ be a morphism of affine gerbes over $\kappa$. The following conditions are equivalent.

\begin{enumerate1}

\itemref{1} The morphism $f$ is locally full.

\itemref{4} If $S$ is an affine scheme over $\kappa$ and $\xi$ and $\eta$ are two objects of\/ $\Gamma(S)$, the induced morphism of fpqc sheaves\/ $\underisom_{\ell}(\xi, \eta) \arr \underisom_{\ell}\bigl(f(\xi), f(\eta)\bigr)$ is surjective.

\itemref{2} The morphism $f$ makes $\Gamma$ into a relative gerbe over $\Delta$.

\itemref{5} If $f$ factors as $\Gamma \arr \Delta' \arr \Delta$, where $\Delta' \arr \Delta$ is a 
faithful homomorphism of affine gerbes, then $\Delta' \arr \Delta$ is an equivalence.

\itemref{6} The pullback homomorphism $\rep\Delta \arr \rep\Gamma$ is fully faithful, and any subrepresentation of a representation of\/ $\Gamma$ in its essential image is also in the essential image.
\end{enumerate1}
\end{proposition}

\begin{proof}
	The equivalence between \refpart{prop:char-locally-full}{1}, \refpart{prop:char-locally-full}{2} and \refpart{prop:char-locally-full}{6} is also established in \cite[Proposition B.2 (2)]{tonini-zhang-fundamental}. We give a complete proof for the convenience of the reader.

\refpart{prop:char-locally-full}{4}$\implies$\refpart{prop:char-locally-full}{1}: it follows from the fact that a homomorphism of affine group schemes $G \arr H$ is faithfully flat if and only if it is an fpqc cover.

\smallskip

\refpart{prop:char-locally-full}{1}$\implies$\refpart{prop:char-locally-full}{4}: if $G \eqdef \underaut_{S}\xi$ and $H \eqdef \underaut_{S}f(\xi)$, then $\underisom_{\ell}(\xi, \eta)$ is a $G$-torsor and $\underisom_{\ell}\bigl(f(\xi), f(\eta)\bigr)$ is an $H$-torsor. The map $f$ induces a homomorphism $G \arr H$ of group scheme over $S$. The map $\underisom_{\ell}(\xi, \eta) \arr \underisom_{\ell}\bigl(f(\xi), f(\eta)\bigr)$ is $G \arr H$-equivariant, so it is enough to show that $G \arr H$ an fpqc cover.

This follows from the definition if $S$ is the spectrum of a field. In the general case, we may pass to an fpqc cover of $S$, and assume that there is a morphism $S \arr \spec \ell$, where $\ell$ is an extension of $\kappa$, and an object $\xi_{0}$ of $\Gamma(\ell)$ whose pullback to $S$ is isomorphic to $\xi$, so the general case follows from the case of a field.

\smallskip

\refpart{prop:char-locally-full}{4}$\iff$\refpart{prop:char-locally-full}{2}: this is straightforward.

\smallskip

\refpart{prop:char-locally-full}{1}$\implies$\refpart{prop:char-locally-full}{2}: we can extend the base field, and assume that $\Gamma = \cB_{\kappa}G$, $\Delta = \cB_{\kappa}H$, and $f$ is induced by a surjective homomorphism of affine group schemes $\phi\colon G \arr H$.

Then the factorization $\Gamma \arr \Delta' \arr \Delta$ corresponds to a factorization $G \arr H' \arr H$, where $H' \arr H$ is a monomorphism. Since $G \arr H$ is an epimorphism, it follows that $H' \arr H$ is an isomorphism, so that $\Delta' \arr \Delta$ is in fact an equivalence.

\smallskip

\refpart{prop:char-locally-full}{2}$\implies$\refpart{prop:char-locally-full}{1}: consider the canonical factorization $\Gamma \arr \Delta' \arr \Delta$. Since $\Delta' \arr \Delta$ is faithful, by hypothesis this is an equivalence, hence $f$ is locally full.

\refpart{prop:char-locally-full}{1}$\iff$\refpart{prop:char-locally-full}{6}: see \cite[III 3.3.2.2]{saavedra}.
\end{proof}

For the following we need the notion of cofiltered system of affined gerbes, and projective limit of such a cofiltered system; for this we refer to \cite[Section~3]{borne-vistoli-fundamental-gerbe}.

\begin{proposition}\label{prop:locally-full}
Let $\{\Delta_{i}\}$ be a cofiltered system of affine gerbes, $\Gamma \arr \projlim \Delta_{i}$ a morphism of affine gerbes. If each composite $\Gamma \arr \projlim \Delta_{i} \arr \Delta_{i}$ is locally full, then $\Gamma \arr \projlim \Delta_{i}$ is also locally full.
\end{proposition}

\begin{proof}
We will use the fact that a homomorphism of affine groups $G \arr H$ over a field $k$ is faithfully flat if and only the corresponding homomorphism of Hopf algebras $k[H] \arr k[G]$ is injective.

Let $\xi$ be an object of $\Gamma(\ell)$, where $\ell$ is an extension of $\kappa$; denote by $\eta_{i}$ the image of $\xi$ in $\Delta_{i}$, $\eta$ its image in $\projlim\Delta_{i}$. Set $G \eqdef \underaut_{\ell}\xi$, $H_{i}\eqdef \underaut_{\ell}\eta_{i}$, $H \eqdef \underaut_{\ell}\eta$. We need to show that the homomorphism $G \arr H$ is faithfully flat, knowing that the composite $G \arr H \arr H_{i}$ is. But $\ell[H] = \indlim_{i}\ell[H_{i}]$; since every homomorphism $\ell[H_{i}] \arr \ell[G]$ 
is injective, the conclusion follows.
\end{proof}

\section{Fibered categories}

Let $\rmp_{X}\colon X \arr\aff{\kappa}$ be a category fibered in groupoids. 

We will consider $X$ as a site with the fpqc topology inherited from $\aff{\kappa}$: a collection $\{\xi_{i} \arr \xi\}$ of arrows in $X$ is an fpqc covering if the corresponding maps $\rmp_{X}\xi_{i} \arr \rmp_{X}\xi$ are flat, and $\rmp_{X}\xi$ is the union of a finite number of images of $\rmp_{X}\xi_{i}$.

The fpqc sheaf $\cO = \cO_{X}$ sends each object $\xi$ into $\cO(U)$, where $U=\rmp_{X}\xi$. If $\kappa'/\kappa$ is a field extension, denote by $X_{\kappa'}$ the fibered product $\aff{\kappa'} \times_{\aff{\kappa}}X$. There exists an obvious homomorphism of $\kappa$-algebras $\H^{0}(X, \cO) \arr \H^{0}(X_{\kappa'}, \cO)$, inducing a homomorphism of $\kappa'$-algebras $\H^{0}(X, \cO)\otimes_{\kappa}\kappa' \arr \H^{0}(X_{\kappa'}, \cO)$.

Recall that a quasi-compact and quasi-separated morphism, or scheme, or algebraic space, is nowadays called \emph{concentrated}.

\begin{definition}
A fibered category $X \arr \aff{\kappa}$ is \emph{concentrated} if there exists an affine scheme $U$ and a representable concentrated faithfully flat morphism $U \arr X$.
\end{definition}

Notice that if $U \arr X$ is as above, and we set $R \eqdef U \times_{X}U$, we obtain an fpqc groupoid $R \double U$ in algebraic spaces, in which $U$ and $R$ are concentrated. (If $X$ is an fpqc stack, which we are not assuming, then $X$ is equivalent to the stack of $(R \double U)$-torsors in the fpqc topology.) From standard arguments in descent theory it follows that we have an exact sequence
   \[
   0 \arr \H^{0}(X, \cO) \arr \H^{0}(U, \cO) \arr \H^{0}(R, \cO)\,.
   \]
From this we easily get the following.

\begin{proposition}
Assume that $X$ is concentrated. For any field extension $\kappa'/\kappa$, the base change homomorphism $\H^{0}(X, \cO)\otimes_{\kappa}\kappa' \arr \H^{0}(X_{\kappa'}, \cO)$ is an isomorphism.
\end{proposition}

\begin{definition}\label{def:reduced}
A fibered category $X$ over $\aff\kappa$ is called \emph{reduced} if every map from $X$ to an algebraic stack $\Gamma$ over $\kappa$ factors through the reduced substack $\Gamma\red \subseteq \Gamma$.

It is \emph{geometrically reduced} if the fibered category $X_{\kappa'} \arr \aff{\kappa'}$ is  reduced for any extension $\kappa'/\kappa$.
\end{definition}

The following is straightforward.

\begin{proposition}
Let $X$ be a fibered category. Suppose that there exists a reduced scheme $U$ and a representable faithfully flat map $U \arr X$. Then $X$ is reduced.
\end{proposition}

So, for example, an affine gerbe $X$ is reduced, because it has a representable faithfully flat map from the spectrum of a field \cite[Proposition 3.1~(b)]{borne-vistoli-fundamental-gerbe}. Since being an affine gerbe is a property that is stable under base change, an affine gerbe is in fact geometrically reduced.

Suppose that $G$ is an affine group scheme over $\kappa$, and $X$ is a fibered category. A $G$-torsor over $X$ is a morphism of fibered categories $X \arr \cB_{\kappa}G$. Morphisms of $G$-torsors are, of course, base-preserving natural transformations. The resulting category of $G$-torsor on $X$ will be denoted by $\tors_{G}(X)$; it is a groupoid. Of course, if $X$ is a scheme then $\tors_{G}(X)$ is equivalent to the categories of classical $G$-torsors over $X$.

A homomorphism $\phi\colon G \arr H$ of affine algebraic group schemes over $\kappa$ yields a functor $\cB_{\kappa}\phi\colon \cB_{\kappa}G \arr \cB_{\kappa}H$, sending $E \arr S$ into $E\times^{G}H \arr S$. Composing with this functor gives a group change functor $\tors_{G}(X) \arr \tors_{H}(X)$; the image of a torsor $E$ will be denoted by $E\times^{G}H$.

If $x_{0} \in X(\kappa)$, a \emph{pointed torsor} over $(X, x_{0})$ will be a pair $(E, e_{0})$, where $E\colon X \arr \cB_{\kappa}G$ is a $G$-torsor, and $e_{0}$ is a $\kappa$-rational point of the $G$-torsor $E(x_{0})$. Pointed $G$-torsors over $(X, x_{0})$ form a category in the obvious way: a morphism $f:(E, e_{0}) \arr (E', e'_{0})$ is a base-preserving natural transformation $f\colon E \arr E'$ such that $f_{x_{0}}(e_{0}) = e'_{0}$. This yields the groupoid $\tors_{G}(X, x_{0})$ of pointed $G$-torsors. 

If $(E, e_{0})$ is a pointed $G$-torsor on $(X, x_{0})$ and $\phi \colon G\arr H$ is a homomorphism of affine group schemes, there is a natural map $E(x_{0}) \arr E(x_{0})\times^{G} H = (E\times^{G}H)(x_{0})$; thus, taking the image of $e_{0}$, the torsor $E\times^{G}H$ becomes a pointed torsor, which we denote by $(E, e_{0})\times^{G}H$. This gives a group change functor $\tors_{G}(X, x_{0}) \arr \tors_{H}(X, x_{0})$.

\section{Fundamental gerbes}\label{sec:fundamental-gerbes}

\begin{definition}\call{def:stable}
Let $\cC$ be a class of affine group schemes of finite type over extensions $\ell$ of $\kappa$; for each $\ell$ we denote by $\cC(\ell)$ the class of group schemes over $\ell$ that are in $\cC$. We say that $\cC$ is \emph{stable} if the following conditions are satisfied.

\begin{enumerate1}

\itemref{1} Each $\cC(\ell)$ is closed under isomorphism of group schemes over $\ell$.

\itemref{2} If $\ell$ is an extension of $\kappa$, $\ell'$ is an extension of $\ell$, and $G$ is a group scheme in $\cC(\ell)$, then $G_{\ell'}$ is in $\cC(\ell')$.

\itemref{3} If $G$ and $H$ are in $\cC(\ell)$, then $G \times_{\ell}H$ is also in $\cC(\ell)$.

\itemref{4} Suppose that $G$ is in $\cC(\ell)$ and $H$ is an $\ell$-subgroup scheme of $G$. Then $H$ is in $\cC(\ell)$.

\itemref{6} Suppose that $G$ is in $\cC(\ell)$
and $H$ is a normal $\ell$-subgroup scheme of $G$. Then $G/H$ is in $\cC(\ell)$.

\itemref{7} If $G$ in $\cC(\ell)$, every inner form of $G$ is in $\cC(\ell)$. \smallskip

\end{enumerate1}
\end{definition}

\begin{definition}
A stable class $\cC$ is said to be \emph{very stable} if whenever $\ell$ is an extension of $\kappa$, $\ell'$ is a finite extension of $\ell$, and $G$ is an affine group scheme of finite type over $\ell$, then $G$ is in $\cC(\ell)$ if and only if $G_{\ell'}$ is in $\cC(\ell')$.

It is called \emph{weakly very stable} if the same is true for all finite separable extensions $\ell'/\ell$.

\end{definition}

\begin{definition}
Let $\cC$ be a stable class. A \emph{pro-$\cC$-group} over $\kappa$ is a group scheme that is a projective limit of groups in $\cC(\kappa)$.
\end{definition}

\begin{definition}
Let $\cC$ be a stable class. A \emph{$\cC$-gerbe over $\kappa$} is an affine gerbe of finite type over $\kappa$ such that for any object $\xi$ in $\Gamma(\ell)$, where $\ell$ is an extension of $\kappa$, the group scheme $\underaut_\ell \xi$ is in $\cC(\ell)$.

A pro-$\cC$-gerbe is a gerbe that is a projective limit of $\cC$-gerbes.
\end{definition}

\begin{remark}
It follows from conditions \refpart{def:stable}{2} and \refpart{def:stable}{7} of Definition~\ref{def:stable} that if $G$ is a group scheme over $\kappa$, then $\cB_{\kappa}G$ is a $\cC$-gerbe if and only if $G$ is in $\cC(\kappa)$.
\end{remark}

\begin{definition}
Let $X$ be a fibered category over $\aff\kappa$. A \emph{$\cC$-fundamental gerbe} $\univ X$ is a pro-$\cC$-gerbe over $\aff\kappa$ with a morphism of fibered categories $X \arr \univ X$ such that for any other pro-$\cC$-gerbe $\Gamma$ the induced morphism
   \[
   \hom_{\kappa}(\univ X, \Gamma) \arr \hom_\kappa(X, \Gamma)
   \]
is an equivalence of categories.
\end{definition}

\begin{remark}
It follows easily from the definition of a projective limit that for $X \arr \univ X$ to be a $\cC$-fundamental gerbe, it is enough to check the condition when $\Gamma$ is a $\cC$-gerbe.
\end{remark}

If $X(\kappa) \neq{\emptyset}$, the concept of a fundamental gerbe can be recast in the more traditional language of groups and torsors.

Let us fix a stable class $\cC$; we will consider $\cC(\kappa)$ as a full subcategory of the category of affine group schemes over $\kappa$. 

\begin{definition}
Let $x_{0}$ be an element of $X(\kappa)$. The pair $(X, x_{0})$ is \emph{$\cC$-rigid} if any pointed $G$-torsor on $(X, x_{0})$, where $G \in \cC(\kappa)$, has trivial automorphism group.
\end{definition}

\begin{definition}
Let $x_{0} \in X(k)$. A \emph{$\cC$-fundamental group} $\fund(X, x_{0})$ is a pro-$\cC$-group which prorepresents the functor $\cC(\kappa)\op \arr \catset$ that sends $G$ into the set of isomorphism classes in $\tors_G(X,x_{0})$.
\end{definition}

\begin{remark}
Clearly, if $\fund(X, x_{0})$ is a $\cC$-fundamental group, we have a canonical bijection between homomorphisms $\fund(X, x_{0}) \arr G$ and isomorphism classes of pointed $G$-torsors on $(X, x_{0})$. This shows that $\fund(X, x_{0})$ is unique, up to a unique isomorphism.
\end{remark}

\begin{proposition}\label{prop:group<->gerbe}
Let $X \arr \aff\kappa$ be a fibered category, $x_{0} \in X(\kappa)$.

Assume that $X$ has a $\cC$-fundamental gerbe $\rho\colon X \arr \univ X$, and denote by $\xi \in \univ X(\kappa)$ be the image of $x_{0}$.

Then $(X, x_{0})$ is $\cC$-rigid, and the pro-$\cC$-group scheme $\underaut_{\kappa}\xi$ is a fundamental group scheme for $X$.

Conversely, assume that $X$ is $\cC$-rigid, and let $\fund(X, x_{0})$ be a $\cC$-fundamental group scheme. Then there exists a morphism $X \arr \cB_{\kappa}\fund(X, x_{0})$ making the gerbe $\cB_{\kappa}\fund(X, x_{0})$ into a $\cC$-fundamental gerbe for $X$.
\end{proposition}

\begin{proof}
If $G$ and $H$ are group schemes over $\kappa$, there is a canonical equivalence between the category of pointed $H$-torsors on $(\cB_{\kappa}G, e)$, where $e \in \cB_{\kappa}G(\kappa)$ is the trivial torsor, and the set of homomorphism $G \arr H$, considered as a category, in which the only arrows are the identities (see \cite[Remarque~1.6.7]{giraud}). 

Suppose that $X$ has a $\cC$-fundamental gerbe $X \arr \univ X$. If we denote by $\xi$ the image of $x_{0}$ in $\univ X$, we have equivalence $\cB_{\kappa}\underaut_{\kappa} \xi\simeq \univ X$, and an isomorphism between the image of the trivial torsor on $\spec \kappa$ with $\xi$. It follows from the definition of a fundamental gerbe that the map $X \arr \univ X$ induces an equivalence between pointed $G$-torsors on $(X, x_{0})$ and pointed $G$-torsors on $(\cB_{\kappa}\underaut_{\kappa} \xi, e)$. Thus $(X, x_{0})$ is $\cC$-rigid, because $\tors_G(X, x_{0})$ is equivalent to a set, and we get a bijection between isomorphism classes in $\tors_G(X, x_{0})$ and $\hom(\underaut_\kappa \xi, G)$.

The other direction is proved with similar arguments.
\end{proof}

For most stable classes $\cC$, fundamental gerbes do not exist in any kind of reasonable generality. The point is the following. Suppose that $\univ X$ exists, and we are given a $2$-commutative diagram
   \[
   \begin{tikzcd}
   X \ar{r}\ar{d} &\Gamma'\ar{d}\\
   \Gamma'' \ar{r} & \Gamma
   \end{tikzcd}
   \]
in which $\Gamma'$, $\Gamma''$ and $\Gamma$ are $\cC$-gerbes. Then the induced morphism $X \arr \Gamma' \times_{\Gamma} \Gamma''$ factors through $\univ X$. However, in many cases it is possible to show that $X \arr \Gamma' \times_{\Gamma} \Gamma''$ cannot factor through a gerbe. Our examples are based on Lemma~\ref{lem:fiber-product}.

\begin{example}\label{ex:does-not-exist}
Assume that $\cC$ contains a reductive non-abelian group. Any such group contains a semisimple non-abelian group $G$, which in turns contains a parabolic subgroup $P$. By Lemma~\ref{lem:fiber-product} we have
   \[
   \cB_{\kappa}P \times_{\cB_{\kappa}G} \spec \kappa \simeq G/P\,;
   \]
hence if $X$ is a projective variety with a non-constant map $X \arr G/P$ (for example $\PP^{1}$), we have a morphism $X \arr \cB_{\kappa}P \times_{\cB_{\kappa}G} \spec \kappa$ that does not factor through a gerbe, because any morphism from a gerbe to $G/P$ factors through $\spec \kappa$, and $X$ cannot have a $\cC$-fundamental gerbe.
\end{example}

\begin{example} For a more subtle example, let $\gm$ act on $\ga$ by multiplication in the usual way. Let $\cC$ be a stable class containing $\gm \ltimes \ga$. In fact, if $\cB_{\kappa}\gm \arr \cB_{\kappa}(\gm \ltimes \ga)$ is induced by the embedding $\gm \subseteq (\gm \ltimes \ga)$, using the lemma above it is easy to see that
   \[
   \cB_{\kappa}\gm \times_{\cB_{\kappa}(\gm \ltimes \ga)} \cB_{\kappa}\gm \simeq 
   [\ga/\gm]\,.
   \]
But a morphism $X \arr [\ga/\gm]$ corresponds to an invertible sheaf $L$ with a section $s \in \H^{0}(X, L)$; if we take a reduced positive-dimensional projective variety $X$, this has an invertible sheaf with a section that vanishes at some points, but not everywhere. This defines a morphism $X \arr [\ga/\gm]$ that does not factor through a gerbe.

\end{example}

\section{Well-founded classes}\label{sec:well-founded}

\subsection{Well-founded actions} Let $G$ be an affine group scheme of finite type over an algebraically closed field $k$, acting on an affine scheme $X$ of finite type over $k$. 

\begin{definition}
A \emph{$G$-reduced subscheme} $V\subseteq X$ is a closed $G$-invariant subscheme of $X$ with the property that every $G$-invariant nilpotent sheaf of ideals in $\cO_{V}$ is $0$.
\end{definition}

Equivalently, a closed $G$-invariant subscheme $V\subseteq G$ is $G$-reduced if the quotient stack $[V/G]$ is reduced.

\begin{remark}
If $G$ is smooth, which is automatically the case when $\cha k = 0$, a closed $G$-invariant subscheme $V \subseteq X$ is $G$-reduced if and only if it is reduced.
\end{remark}

\begin{definition}\label{def:well-founded-action}
The action of $G$ on $X$ is \emph{well-founded} if for any $G$-reduced subscheme $V \subseteq X$ such that $k[V]^{G} = k$, the action of $G(k)$ on $V(k)$ is transitive.
\end{definition}

\begin{remark}
Being well-founded is a property of the quotient stack $[X/G]$: the action is well founded if for every closed reduced substack $\cV \subseteq [X/G]$ with $\H^{0}(\cV, \cO) = k$, the groupoid $\cV(k)$ is transitive (or, equivalently, $\cV$ is a gerbe over $k$).
\end{remark}

\begin{remark}
If the action of $G$ is well-founded, then it has closed orbits. The converse holds if $G$ is geometrically reductive, by geometric invariant theory, but not in general (consider the example in which $X$ is a reductive group and $G$ is a parabolic subgroup acting by translation).
\end{remark}

It is immediate to give examples of actions that are not well-founded: the action of $\gm$ on $\AA^{1}$ by multiplication springs to mind.

A class of examples of well-founded actions comes from the following proposition.
Recall that a linear group scheme $G$ is called \emph{unipotent} if for every non-zero representation $G \arr \GL(V)$ we have $V^{G} \neq 0$.

\begin{proposition}\label{prop:unipotent->well-founded}
If $G$ is unipotent, the action of $G$ on $X$ is always well-founded.
\end{proposition}

\begin{proof}
Assume that it is not so. Let $V\subseteq X$ be an invariant closed subset such that $k[V]^{G} = k$ (we do not need to assume that $V$ is $G$-reduced), and $V(k)$ contains more than one $G(k)$-orbit. If $v_{0} \in V(k)$ is a point whose orbit has minimal dimension, then its scheme-theoretic orbit $\Omega \subseteq V$ is closed. From the lemma below it follows that $\Omega = V$, and the conclusion follows.
\end{proof}

\begin{lemma}
Let $V$ be an affine $k$-scheme with an action of a unipotent $k$-group scheme $G$. Assume that $\H^{0}(V, \cO)^{G} = k$. Then the only proper $G$-invariant closed subscheme of $V$ is $\emptyset$.
\end{lemma}

\begin{proof}
Let $W \subseteq V$ be a proper invariant closed subscheme; call $I \subseteq \H^{0}(V, \cO)$ its ideal. Then $G$ acts rationally on $I \neq 0$; so there exists $f \in I^{G} \setminus \{0\}$. Since $\H^{0}(V, \cO)^{G} = k$ we have $f \in k \setminus \{0\}$. So $f$ is invertible, therefore $W = \emptyset$.
\end{proof}

In what follows we will use the following notation. Assume $\cha k = p > 0$, and $n \in \NN$. Let $Y$ be a $k$-scheme. We denote by $Y_{n}$ the scheme $Y$, considered as a $k$-scheme via the composite $Y \arr \spec k \arr \spec k$, where the homomorphism $\spec k \arr \spec k$ is induced by the ring homomorphism $k \arr k$ defined by $x \mapsto x^{p^{-n}}$. Equivalently we can defined $Y_{n}$ as the fibered product $\spec k \times_{\spec k} Y$, where the map $\spec k \arr \spec k$ is induced by $x \mapsto x^{p^{n}}$. We have a relative Frobenius map $F_{n}\colon Y_{n} \arr Y$.

\begin{lemma}\label{lem:criterion-well-founded-action-1}
Suppose that the induced action of $G\red$ on $X\red$ is well-founded.  Then the action of $G$ on $X$ is also well-founded.
\end{lemma}

\begin{proof}
If $\cha k = 0$ then $G = G\red$ is smooth, and the result is obvious, since $X\red$ and $X$ have the same closed reduced subschemes.

Assume $\cha k = p > 0$; assume that $V \subseteq X$ is a $G$-reduced subscheme of $X$ with $k[V]^{G} = k$. It is enough to show that $k[V]^{G\red} = k$, for then the action of $G(k) = G\red(k)$ on $V(k) = V\red(k)$ will be transitive.

Fix a positive integer $n$: the Frobenius morphism $F_{n}\colon X_{n} \arr X$ is $G_{n} \arr G$-equivariant, and carries the closed subscheme $V_{n} \subseteq X_{n}$ in $V$. For $n \gg 0$, the scheme theoretic images of $V_{n}$ in $V$ and of $G_{n}$ in $G$ will be, respectively, $V\red$ and $G\red$; hence $k[V\red]^{G\red} \subseteq k[V_{n}]^{G_{n}} = k$, and this completes the proof.
\end{proof}

\begin{lemma}\label{lem:criterion-well-founded-action-2}
Assume that $G$ is smooth, and that the action of $G^{0}$ on $X$ is well-founded. Then the action of $G$ on $X$ is also well-founded.
\end{lemma}

\begin{proof}
Let $\Gamma \eqdef G/G^{0}$. Let $V \subseteq X$ be a $G$-reduced subscheme such that $k[V]^{G} = k$. The action of $\Gamma$ on the connected components of $V$ is clearly transitive. Let $W$ be a connected component of $V$, and let $\Delta \subseteq \Gamma$ be the stabilizer of $W$. Then we have $(k[W]^{G^{0}})^{\Delta} = (k[V]^{G^{0}})^{\Gamma} = k[V]^{G} = k$; hence $k[W]^{G^{0}}$ is an integral extension of $k$, so it is contained in the integral closure of $k$ in $k[W]$, which equals $k$. Since the action of $G^{0}$ is well-founded by hypothesis, we have that the action of $G^{0}(k)$ on $W(k)$ is transitive, which implies that the action of $G(k)$ on $V(k)$ is transitive.
\end{proof}

Lemmas \ref{lem:criterion-well-founded-action-1} and \ref{lem:criterion-well-founded-action-2} imply the following useful fact.

\begin{proposition}\label{prop:criterion-well-founded-action}
Suppose that the induced action of $G\ored$ on $X\red$ is well-found\-ed. Then the action of $G$ on $X$ is also well-founded.
\end{proposition}

\subsection{Well-founded group schemes}

\begin{definition}\label{def:well-founded-group}
An affine group scheme of finite type $G$ over a field $k$ is called \emph{well-founded} if for any algebraically closed extension $\ell$ of $k$ and any two subgroup schemes $H$ and $K$ of $G_{\ell}$, the action of $H \times K$ on $G_{\ell}$ defined by $(h,k)\cdot g = h^{-1}gk$ is well-founded.
\end{definition}

This strange-looking definition is exactly what is needed to make the proof of Theorem~\ref{thm:maintheorem} work. However, it turns out to be equivalent to the following much more natural condition.

Recall the following facts from the theory of algebraic groups. Assume that $k$ is algebraically closed, and let $G$ be a smooth connected affine algebraic group over $k$.

\begin{enumerate1} 

\item $G$ is solvable if and only if $G$ contains no non-trivial semisimple subgroups, and if and only if it is a semidirect product $U \rtimes T$, where $T$ is a torus and $U$ is a smooth unipotent group.

\item $G$ is nilpotent if and only if it is the product of a torus and a unipotent group.

\end{enumerate1}

\begin{definition} Let $G$ be an affine group scheme of finite type over a field $k$.

We say that $G$ is \emph{virtually nilpotent} if $(G_{\overline{k}})\ored$ is nilpotent.

\emph{Virtually abelian} and \emph{virtually unipotent} group schemes are defined similarly.

\end{definition}

Our main result in this section is the following characterization of well-founded group schemes.

\begin{theorem}\label{thm:char-well-founded}
An affine group scheme of finite type over $k$ is well-founded if and only if it is virtually nilpotent.
\end{theorem}

Equivalently, affine group scheme $G$ of finite type over $k$ is well-founded if and only if $(G_{\overline{k}})\ored$ is the product of a unipotent group and a torus.

\begin{proof}
For the proof we will need the following facts. 

\begin{lemma}\call{lem:well-founded-properties} Assume that  $G$ is an affine group scheme of finite type over $k$.

\begin{enumerate1}

\itemref{3} 
If $G$ is well-founded, subgroups and quotients of $G$ are well-founded.

\itemref{4} If $\ell$ is an extension of $k$, then $G_{\ell}$ is well-founded if and only if $G$ is well-founded.

\itemref{1} $G$ is well-founded if and only if the condition of Definition~\ref{def:well-founded-group} is satisfied when $H$ and $K$ are smooth and connected.

\itemref{2} $G$ is well-founded if and only if $(G_{\overline{k}})\ored$ is well-founded.

\end{enumerate1}
\end{lemma}

\begin{proof}
\refpart{lem:well-founded-properties}{3} and \refpart{lem:well-founded-properties}{4} are straightforward.

\refpart{lem:well-founded-properties}{1} follows from Proposition~\ref{prop:criterion-well-founded-action}.

Let us prove \refpart{lem:well-founded-properties}{2}. We may assume that $k$ is algebraically closed. If $G$ is well-founded, so is $G\ored$, by part~\refpart{lem:well-founded-properties}{3}.

Assume that $G\ored$ is well-founded. Let $H$ and $K$ be subgroup schemes of $G$, and let us check that the action of $H \times K$ on $G$ is well-founded. By Proposition~\ref{prop:criterion-well-founded-action} we may assume that $G$ is smooth, and $H$ and $K$ are smooth and connected.

Let $V \subseteq G$ be a closed $H\times K$-invariant subscheme with $k[V]^{H \times K} = k$. Then $V$ must be connected. If $V \subseteq G^{0}$, then the action of $H \times K$ on $V$ is transitive by hypothesis. Let us reduce to this case.

Let $g \in V(k)$, and set $W \eqdef g^{-1}V$. Then $W$ is contained in $G^{0}$. It is also invariant under the action of $(g^{-1}Hg)\times K$, and $k[W]^{(g^{-1}Hg)\times K} \simeq k[V]^{H\times K} = k$. So the action of $(g^{-1}Hg)\times K$ on $W$ is transitive, which implies that the action of $H \times K$ on $V$ is transitive.
\end{proof}

Let us prove the Theorem. By Lemma~\ref{lem:well-founded-properties}, we may assume that $k$ is algebraically closed, and $G$ is smooth and connected. 

Now assume that $G$ is well-founded. We want to show that $G$ is virtually 
nilpotent.

If $G$ were not solvable, it would contain a non-trivial semisimple group. By 
Lemma~\refall{lem:well-founded-properties}{3}, we may assume that $G$ is non-trivial and semisimple. Let $P$ be a parabolic subgroup of $G$: by setting $H = P$ and $K = \{1\}$ we see that $k[G]^{H \times K} = k$, but the action of $H \times K$ is clearly not transitive. This is a contradiction.

So $G$ is a semidirect product $U \rtimes T$, where $T$ is a torus and $U$ is a smooth unipotent group. We need to show that the action of $T$ on $U$ is trivial. Take $H = K = T$: the quotient $G/T$ is isomorphic to $U$, and the corresponding left action on $U$ is given by conjugation. By hypothesis, the orbits of the action of $T \times T$ on $G$ are closed: this implies that the orbits for the action of $T$ on $U$ are closed. Set $U\git T \eqdef \spec k[U]^{T}$: by affine GIT, the fibers of the projection $U\arr U\git T$ are, set-theoretically, precisely the orbits of $T$ on $U$. But the identity in $U$ forms an orbit: this implies that $\dim U = \dim U\git T$, so the generic orbit of $T$ on $U$ is finite. Since $T$ is smooth and connected, this implies that the action is trivial.

Now we need to prove that a product $U \times T$, where $U$ is smooth 
unipotent and $T$ is a torus, is well-founded. 

First, assume that $G$ is abelian. Let $H$ and $K$ be subgroup schemes of $G$. Then $H \backslash G/K = G/L = G \git L$, where $L$ is the subgroup scheme of $G$ generated by $H$ and $K$, and the statement is easy.

This takes care of the case of the torus. The case that $G$ is unipotent follows from Proposition~\ref{prop:unipotent->well-founded}.

For the general case, let $H$ be a smooth connected subgroup scheme of $G$. Since $H$ is also nilpotent, we can split 
it as 
product $H_{U} \times H_{T}$ of a unipotent group scheme and a torus: but since all homomorphisms $H_{U} \arr T$ and $H_{T} \arr U$ are trivial, we have that $H_{U} \subseteq U$ and $H_{T} \subseteq T$.

Let $H$ and $K$ be smooth connected subgroup schemes of $G$, and let $V \subseteq G$ be a 
$(H \times K)$-reduced 
subscheme of $G$ such that $k[V]^{H \times K} = k$. Let $V_{U}$ and $V_{T}$ be the scheme-theoretic images of $V$ in $U$ and $T$ respectively. Then  $V \subseteq V_{U} \times V_{T}$; furthermore, $V_{U}$ and $V_{T}$ are reduced, invariant under the actions of 
$H_{U}\times K_{U}$  and $H_{T}\times K_{T}$ respectively. We have
   \[
   k[V_{U}]^{H_{U} \times K_{U}} \subseteq k[V_{U}]^{H \times K}
   \subseteq k[V]^{H \times K} = k\,,
   \]
so that $k[V_{U}]^{H_{U} \times K_{U}} = k$, and analogously $k[V_{T}]^{H_{T} \times K_{T}} = k$. Since $U$ and $T$ are well-founded, we have that $V_{U}$ and $V_{T}$ are orbits for $H_{U} \times K_{U}$ and $H_{T} \times K_{T}$ respectively. It follows immediately that the action of $H \times T$ on $V$ is transitive. This completes the proof of this implication, and of Theorem~\ref{thm:char-well-founded}.
\end{proof}

We will need the following easy fact, which we record here.

\begin{lemma}
Let $G$ be a well-founded group scheme over $k$, and let $\phi\colon H \arr G$ and $\psi\colon K \arr G$ two homomorphisms of affine group schemes. Then the action of $H \times K$ on $G$ defined by $g\cdot (h, k) = \phi(h)^{-1}g\psi(k)$ is well founded.
\end{lemma}

\begin{proof}
Replace $H$ and $K$ with their images in $G$.
\end{proof}

\subsection{Well-founded classes}

\begin{definition}\call{def:well-founded-class}
A stable class $\cC$ is \emph{well-founded} if it consists of virtually nilpotent groups.
\end{definition}

The class of all virtually nilpotent groups is stable, hence well-founded. Obviously, if $\cC$ is a well-founded class, any subclass of $\cC$ that is closed under isomorphisms, extensions of scalars, and taking products, subgroups, quotients and twisted forms, is also well-founded.

This yields a vast range of examples of well-founded classes.

\begin{enumerate1}\call{list-examples}

\itemref{vf} All virtually nilpotent group schemes.

\itemref{a} Virtually abelian group schemes.

\itemref{b} Virtually unipotent group schemes.

\itemref{1} Finite group schemes.

\itemref{2} Linearly reductive finite group schemes.

\itemref{3} Abelian affine group schemes.

\itemref{4} Diagonalizable group schemes.

\itemref{5} Group schemes of multiplicative type.

\itemref{6} Unipotent group schemes.

\end{enumerate1}

All these classes are in fact very stable, with the exception of \refpart{list-examples}{4}.

\section{Existence of fundamental gerbes}\label{sec:maintheorem}

The following is our first main theorem.

\begin{theorem}\label{thm:maintheorem}
Let $X$ be a fibered category over $\aff\kappa$. Assume that $X$ is concentrated and geometrically reduced, and that $\H^{0}(X, \cO) = \kappa$. If $\cC$ is a well-founded class, then there exists a fundamental gerbe $X \arr \univ X$.
\end{theorem}

\begin{corollary}
Let $X$ be fibered category over $\aff\kappa$, and let $x_{0} \in X(k)$. Assume that $X$ is concentrated and geometrically reduced, and that $\H^{0}(X, \cO) = \kappa$. If $\cC$ is a well-founded class, then there exists a $\cC$-fundamental group $\fund(X, x_{0})$.
\end{corollary}

\begin{remark}
	By definition, the fundamental gerbe $\univ X$ only depends on the class of $\cC$-gerbes defined over $\kappa$; there maybe different classes $\cC$ for which the class of $\cC$-gerbes defined over $\kappa$ coincide (for example, when $\kappa$ is algebraically closed, $\cC$-gerbes over $\kappa$ are the same when $\cC$ is either the class of diagonalizable groups, or of groups schemes of multiplicative type).
\end{remark}

\begin{remark}\label{rmk:necessary}
There is a kind of converse of Theorem~\ref{thm:maintheorem}. Suppose that $\kappa$ is algebraically closed, and that $\cC$ is a stable class such that $\univ X$ exists for all $X \arr \aff\kappa$ satisfying the conditions of Theorem~\ref{thm:maintheorem}. Let $G$ be in $\cC(\kappa)$, and let $H$ and $K$ be subgroup schemes of $G$. Then we claim that the action of $H\times K$ on $G$ is well-founded.

By Lemma~\ref{lem:fiber-product} we have $\cB_{\kappa}H\times_{{\cB}_{\kappa}G} \cB_{\kappa}K = [G/H\times K]$. Let $V \subseteq G$ be an $H\times K$-reduced subscheme with $\kappa[V]^{H \times K} = \kappa$; then $X \eqdef [V/H \times K]$ satisfies the hypotheses of Theorem~\ref{thm:maintheorem}. The closed embedding
   \[
   X \subseteq [G/H\times K] = \cB_{\kappa}H\times_{{\cB}_{\kappa}G} \cB_{\kappa}K
   \]
must factor through $\univ X$, and this clearly implies that $X(\kappa)$ contains a unique isomorphism class, so the action of $(H \times K)(\kappa)$ on $V(\kappa)$ is transitive.

In fact, one can prove that to conclude that the action of $H\times K$ on $G$ is well-founded it is enough to assume that $\univ X$ exists for all schemes satisfying the conditions of Theorem~\ref{thm:maintheorem}. The proof is somewhat complicated, and we omit it.
\end{remark}

The strategy of the proof of the theorem is exactly the same as that in \cite{borne-vistoli-fundamental-gerbe}.

\begin{definition}
Let $\Gamma$ be a $\cC$-gerbe. A morphism of fibered categories $X \arr \Gamma$ is \emph{Nori-reduced} if for any factorization $X \arr \Gamma' \arr \Gamma$, where $\Gamma'$ is a $\cC$-gerbe and $\Gamma' \arr \Gamma$ is faithful, then $\Gamma' \arr \Gamma$ is an isomorphism.
\end{definition}

Let $I$ be a skeleton of the $2$-category of Nori reduced morphisms $X \arr \Gamma$. Thus, an element $i$ of $I$ is a Nori-reduced morphism $X \arr \Gamma_{i}$, and an arrow $u\colon j \arr i$ is given by a $2$-commutative diagram
   \[
   \begin{tikzcd}[row sep=5pt]
   &\Gamma_{j}\ar{dd}{u}\\
   X\ar{ur}\ar{dr}\\
   &\Gamma_{i}\,.
   \end{tikzcd}
   \]

The fundamental gerbe $\univ X$ will be constructed as the limit $\projlim_{i \in I}\Gamma_{i}$; for this we need to show that $I$ is a boolean cofiltered $2$-category, and that every morphism from $X$ to a $\cC$-gerbe $\Gamma$ is bounded by a Nori-reduced morphism $X \arr \Gamma_i$.

Let us recall from \cite{borne-vistoli-fundamental-gerbe} the notion of \emph{scheme-theoretic image} of a morphism of fibered categories $f\colon X \arr Y$, where $X$ is a concentrated fibered category and $Y$ is an algebraic stack. We define the scheme-theoretic image $Y'$ of $X$ in $Y$ to be the intersection of all the closed substacks $Z$ of $Y$ such that $f$ factors, necessarily uniquely, though $Z$. Alternatively, $Y' \subseteq Y$ is the closed substack associated with the kernel of the natural homomorphism $\cO_{Y} \arr f_{*}\cO_{X}$. The fact that this is quasi-coherent is easy when $X$ is a concentrated scheme; the general case reduces to this by using an fpqc cover $U \arr X$.

The key point is the following.

\begin{lemma}\label{lem:basic}
Let\/ $\Gamma' \arr \Gamma$ and\/ $\Gamma'' \arr \Gamma$ be two morphisms of\/ $\cC$-gerbes. The scheme-theoretic image of any morphism $X \arr \Gamma' \times_{\Gamma} \Gamma''$, where $X\neq \emptyset$, is a $\cC$-gerbe.
\end{lemma}

\begin{proof}
Call $\Delta \subseteq \Gamma' \times_{\Gamma} \Gamma''$ the scheme-theoretic image of $X$. We need to show that $\Delta$ is a gerbe, and that if $\eta$ is in $\Delta(\ell)$, where $\ell$ is an extension of $\kappa$, the group scheme $\underaut_{\ell}\eta$ is in $\cC$.

For this we can base change to $\ell$, and assume that $\kappa = \ell$. Call $\xi'$, $\xi''$ and $\xi$ the images of $\eta$ in $\Gamma'$, $\Gamma''$ and $\Gamma$ respectively, and set $G' \eqdef \underaut_{\kappa}\xi'$, and analogously for $G''$ and $G$. The morphisms $\Gamma' \arr \Gamma$ and $\Gamma'' \arr \Gamma$ induce homomorphisms $\phi'\colon G' \arr G$ and $\phi''\colon G'' \arr G$; furthermore we have $\underaut_\kappa \eta = G' \times_{G} G''$, and since $G' \times_{G} G''$ is a subgroup scheme of $G' \times G''$ we have $G' \times_{G} G'' \in \cC$.

Let us check that $\Delta$ is a gerbe. We can make a further extension of $\kappa$, and assume that $\kappa$ is algebraically closed. There are isomorphisms $\Gamma \simeq \cB_{\kappa}G$, $\Gamma' \simeq \cB_{\kappa}G'$ and $\Gamma'' \simeq \cB_{\kappa}G''$, and the morphisms $\Gamma' \arr \Gamma$ and $\Gamma'' \arr \Gamma$ are induced by $\phi'$ and $\phi''$. By Lemma~\ref{lem:fiber-product} we have an isomorphism
   \[
   \cB_{\kappa}G' \times_{\cB_{\kappa}G} \cB_{\kappa}G'' \simeq 
   [G/(G' \times G'')]
   \]
where the action of $G' \times G''$ on $G$ is defined by
   \[
   g\cdot (g', g'') \eqdef \phi'(g')^{-1} g \phi''(g'')\,.
   \]
The scheme-theoretic image of $X$ in $[G/(G' \times G'')]$ is of the form $[V/(G' \times G'')]$, where $V$ is a $(G' \times G'')$-stable subscheme of $G$. Call $A$ the $\kappa$-algebra corresponding to $V$, so that $V\git(G' \times G'')= \spec A^{G' \times G''}$. Since $H^{0}(X, \cO) = \kappa$, we have that the composite
$X \arr [V/(G' \times G'')] \arr V\git(G' \times G'')$ factors through $\spec \kappa$; but since the scheme-theoretic image of $X$ in $[V/(G' \times G'')]$ is $[V/(G' \times G'')]$ itself, we see that $A^{G' \times G''} = \kappa$. Then it follows that the action of $H \eqdef G' \times G''$ on $V$ has only one orbit (this is immediately seen by replacing $G'$ and $G''$ with their images in $G$). Choose a point $v_{0} \in V(\kappa)$, and call $K \subseteq H$ the stabilizer of $v_{0}$; we have a closed embedding $\Omega \eqdef H/K \subseteq V$. Call $I$ the sheaf of ideals of $[\Omega/H]$ in $[V/H]$; the inverse image of $I$ in $\cO_{V}$ is nilpotent, so the inverse image of $I$ in $X$ is $0$, because $X$ is reduced. Hence $X \arr [V/H]$ factors through $[\Omega/H] = [(H/K)/H] = \cB_{\kappa}K$. So $[V/H] = \cB_{\kappa}K$, and this ends the proof of the lemma.
\end{proof}

\begin{lemma}\label{lem:dominance}
If $\Delta$ is a $\cC$-gerbe, every morphism $X \arr \Delta$ factors as $X \arr \Gamma$, where $\Gamma$ is a $\cC$-gerbe, $X \arr \Gamma$ is Nori-reduced and $\Gamma \arr \Delta$ is representable.
\end{lemma}

\begin{proof}
If $X \arr \Delta$ is not Nori-reduced, choose a factorization $X \arr \Delta_{1} \arr \Delta$, where $\Delta_{1} \arr \Delta$ is representable. If $X \arr \Delta_{1}$ is not Nori-reduced, let us choose an analogous factorization $X \arr \Delta_{2} \arr \Delta_{1}$. This process can not continue ad infinitum, as it follow from the following lemma.

\begin{lemma}
If $\Delta$ is an affine gerbe, and
    \[
   \cdots \arr \Delta_{i} \arr \Delta_{i-1} \arr \cdots \arr \Delta_{1} \arr \Delta
   \]
is an infinite sequence of representable maps of affine gerbes, there exists a positive integer $m$ such that $\Delta_{i} \arr \Delta_{i-1}$ is an equivalence for all $i \geq m$.
\end{lemma}

\begin{proof}
We can extend the base field $\kappa$, and assume that $\kappa$ is algebraically closed. If $\ell$ is an algebraically closed extension of $\kappa$ such that $\Delta_{i}(\ell) \neq \emptyset$, choose an object $\xi$ of $\Delta_{i}(\ell)$. Call $d_{i}$ the dimension of $\aut_{\ell}\xi$, and $e_{i}$ the degree of the finite scheme $\aut_{\ell}\xi/(\aut_{\ell}\xi)\ored$. It is immediate to show that $d_{i}$ and $e_{i}$ are independent of $\xi$ and $\ell$. If $\eta$ denotes the image of $\xi$ in $\Delta_{i-1}(\ell)$, the morphism $\aut_{\ell}\xi \arr \aut_{\ell}\eta$ induced by the map $\Delta_{i} \arr \Delta_{i-1}$ is a monomorphism, and an isomorphism if and only if $\Delta_{i} \arr \Delta_{i-1}$ is an equivalence.

Clearly $d_{i} \leq d_{i-1}$, and if $d_{i} = d_{i-1}$ then $e_{i}\leq e_{i-1}$. Furthermore, $d_{i} = d_{i-1}$ and $e_{i} = e_{i-1}$ if and only if $\aut_{\ell}\xi \arr \aut_{\ell}\eta$ is an isomorphism, and the statement follows from this.
\end{proof}

This ends the proof of Lemma~\ref{lem:dominance}.
\end{proof}

Let us show that $I$ is a cofiltered category; we need to show that given two arrows $j \arr i$ and $k \arr i$ in $I$, there exists a commutative diagram
   \[
   \begin{tikzcd}[row sep=-3pt]
   &j\ar{rd}&\\
   l\ar{ur}\ar{dr}&&i\,,\\
   &k\ar{ru}&\\
   \end{tikzcd}
   \]
and, that given two arrows $u$, $v\colon j \arr i$, there exists a unique $2$-arrow $u \arr v$.

For the first point, the two arrows $j \arr i$ and $k \arr i$ correspond to a $2$-commutative diagram
   \[
   \begin{tikzcd}[row sep=7pt]
   &\Gamma_{j}\ar{rd}&\\
   X\ar{ur}\ar{dr} \ar{rr}&&\Gamma_{i}\\\
   &\Gamma_{k}\ar{ru}&\\
   \end{tikzcd}
   \]
inducing a morphism $X \arr \Gamma_{j}\times_{\Gamma_{k}}\Gamma_{i}$. By Lemma~\ref{lem:basic} this factors through a $\cC$-subgerbe $\Delta \subseteq \Gamma_{j}\times_{\Gamma_{k}}\Gamma_{i}$, and by Lemma~\ref{lem:dominance} it factors through a Nori-reduced morphism $X \arr \Gamma_{l}$, proving what we want.

The second fact follows from the analogue of \cite[Lemma~5.13]{borne-vistoli-fundamental-gerbe}. 

\begin{lemma}\label{lem:unique-arrow}
Let $f\colon X \arr \Gamma$ and $g\colon X \arr \Delta$ be morphisms, where $\Gamma$ and $\Delta$ are $\cC$-gerbes and $f$ is Nori-reduced. Suppose that $u$, $v\colon \Gamma \arr \Delta$ are morphisms of fibered categories, and $\alpha\colon u \circ f \simeq g$ and $\beta\colon v \circ f \simeq g$ are isomorphisms. Then there exists a unique isomorphism $\gamma\colon u \simeq v$ such that $ \beta\circ(\gamma * \id_{f})=\alpha  $.
\end{lemma}

This can be expressed by saying that, given two 2-commutative diagrams
   \[
   \begin{tikzcd}[row sep=5pt]
   &\Gamma\ar{dd}{u}\\
   X\ar{ur}{f}\ar[swap]{dr}{g}\\
   &\Delta
   \end{tikzcd}
   \qquad
   \text{and}
   \qquad
   \begin{tikzcd}[row sep=5pt]
   &\Gamma\ar{dd}{v}\\
   X\ar{ur}{f}\ar[swap]{dr}{g}\\
   &\Delta
   \end{tikzcd}
   \]
in which $f$ is Nori-reduced, there exists a unique isomorphism $u \simeq v$ making the diagram
   \[
   \begin{tikzcd}[column sep=30pt, row sep=8pt]
   &\Gamma\ar[bend left=25]{dd}{v}\ar[bend right=25, swap]{dd}{u}\\
   X\ar{ur}{f}\ar[swap]{dr}{g}\\
   &\Delta
   \end{tikzcd}
   \]
2-commutative.

\begin{proof}
This is virtually identical to the proof of \cite[Lemma~5.13]{borne-vistoli-fundamental-gerbe}.

Consider the category $\Gamma' \arr \Gamma$ fibered in sets over $\Gamma$, whose objects over a $\kappa$-scheme $T$ are pairs $(\xi, \rho)$, where $\xi$ is an object of $\Gamma(T)$ and $\rho$ is an isomorphism of $u(\xi)$ with $v(\xi)$ in $\Delta(T)$. This can be written as a fibered product
   \[
   \begin{tikzcd}
   \Gamma' \rar\dar & \Gamma \dar{\generate{u,v}}\\
   \Delta \rar{\delta} & \Delta \times \Delta\,,   
   \end{tikzcd}
   \]
where $\delta\colon \Delta \arr \Delta \times \Delta$ is the diagonal. 

An isomorphism $u \simeq v$ corresponds to a section of the projection $\Gamma' \arr \Gamma$, or, again, to a substack $\Gamma'' \subseteq \Gamma'$ such that the restriction $\Gamma'' \arr \Gamma$ of the projection is an isomorphism. The composite isomorphism $u \circ f \xarr{\alpha} g \xarr{\beta^{-1}} v \circ f$ yields a lifting $X \arr \Gamma'$ of $f\colon X \arr \Gamma$; the thesis can be translated into the condition that there exists a unique substack $\Gamma'' \subseteq \Gamma'$ as above, such that $X \arr \Gamma'$ factors through $\Gamma''$. By Lemma~\ref{lem:basic}, there is a unique closed substack $\Gamma''$ of $\Gamma'$ that is a gerbe, such that $X \arr \Gamma'$ factors through $\Gamma''$. However, $\Gamma'' \arr \Gamma$ is representable, because $\Gamma' \arr \Gamma$ is, so $\Gamma'' \arr \Gamma$ is an equivalence, since $f$ is Nori-reduced.
\end{proof}

Now we set $\univ X \eqdef \projlim_{i \in I} \Gamma_{i}$; the morphisms $X \arr \Gamma_{i}$ yield a morphism $X \arr \univ X$. We need to show that if $\Delta$ is a $\cC$-gerbe, the induced functor
   \[
   \hom_{\kappa}(\univ X, \Delta) \arr \hom_{\kappa}(X, \Delta)
   \]
is an equivalence. From Lemma~\ref{lem:dominance} it follows that it is essentially surjective.

Now, the natural functor
   \[
   \indlim_{i\in I}\hom_{\kappa}(\Gamma_{i}, \Delta) \arr \hom_{\kappa}(\univ X, \Delta)
   \]
is an equivalence, by \cite[Proposition~3.7]{borne-vistoli-fundamental-gerbe}; hence to prove that the above functor is fully faithful it is enough to show that for every $i \in I$ the induced functor $\hom(\Gamma_{i}, \Delta) \arr \hom(X, \Delta)$ is fully faithful. 

Call $f\colon X \arr \Gamma_{i}$ the canonical morphism. Consider two morphisms $u$, $v\colon \Gamma_{i} \arr \Delta$, and an isomorphism of functors $\beta\colon u\circ f \simeq v \circ f$. We need to check that there exists a unique isomorphism $\gamma\colon u \simeq v$ such that $\gamma * \id_{f} = \beta$. This follows immediately from Lemma~\ref{lem:basic}.

\section{Change of class}\label{sec:change-class}

Let $\cC$ be a well-founded class of group schemes over $\kappa$, and let $\Gamma$ be an affine gerbe. It follows from Proposition~\ref{prop:char-locally-full} that a morphism $\Gamma \arr \Delta$ is Nori-reduced if and only if it is locally full. From this, and from Proposition~\ref{prop:locally-full} we deduce that the morphism $\Gamma \arr \univ{\Gamma}$ is locally full.

From now on we use the notation $\Gamma^{\cC}$ for $\univ\Gamma$. Another way of thinking about $\Gamma^{\cC}$ is the following. We can write $\Gamma$ as the projective limit $\projlim \Gamma_{i}$, where the limit is taken over all the locally fully morphisms $\Gamma \arr \Gamma_{i}$. Let $\{\Gamma_{j}\}_{j \in J}$ denote the full subcategory of $I$ consisting of those $\Gamma_{i}$ that are $\cC$-gerbes. This is a cofiltered full subcategory of $I$. Every morphism from $\Gamma$ to a $\cC$-gerbe factors uniquely through $\projlim \Gamma_{j}$; hence $\Gamma^{\cC} = \projlim \Delta_{i}$.

Suppose that $\cD$ is a stable subclass of $\cC$. Since a pro-$\cD$-gerbe is also a pro-$\cC$-gerbe, the morphism $X \arr \univd X$ induces a morphism $\univ X \arr \univd X$.

\begin{proposition}\label{prop:induced-locally-full}
The induced morphism  $\univ{X} \arr \univd{X}$ is locally full.
\end{proposition}

\begin{proof}
Clearly the morphism $X \arr \univ{X}$ induces an equivalence between $\univd{X}$ and $\univd{\univ{X}}$, so the result follows from the previous discussion.
\end{proof}

\begin{corollary}\label{cor:fully-faithful}
If\/ $\rep{\univ{X}} \arr \vect{X}$ is fully faithful, so is $\rep{\univd{X}} \arr \vect{X}$.
\end{corollary}

\begin{proof}
This follow from Propositions \ref{prop:induced-locally-full} and \ref{prop:char-locally-full}.
\end{proof}

\begin{corollary}\label{cor:finite->realizable}
If $\cC$ is a stable class of finite group schemes, the pullback functor $\rep\univ X \arr \vect_{X}$ is fully faithful.
\end{corollary}

\begin{proof}
If $\cC$ is the whole class of finite group schemes this is in \cite[Theorem~7.9]{borne-vistoli-fundamental-gerbe}. The general case follows from Corollary~\ref{cor:fully-faithful}.
\end{proof}

\section{Weil restriction and change of base}\label{sec:base-change}

If $A$ is a finite $\kappa$-algebra, and $X \arr \aff A$ a fibered category, we have the Weil restriction $\rR_{A/\kappa}X\arr \aff\kappa$; for the definition and the basic properties we refer to \cite[Section~6]{borne-vistoli-fundamental-gerbe}.

We will use the following fact. 

\begin{lemma}\label{lem:weil-transfer-classifying}
Let $G$ be an affine group scheme over a finite extension $\ell$ of $\kappa$. Then the Weil restriction $\rR_{\ell/\kappa}(\cB_{\ell}G)$ is canonically equivalent to $\cB_\kappa(\rR_{\ell/\kappa}G)$.
\end{lemma}

\begin{proof}
Let $T$ be an affine $\kappa$-scheme, $T \arr \rR_{\ell/\kappa}(\cB_{\ell}G)$ a morphism, corresponding to a morphism $T_{\ell} \arr \cB_{\ell}G$, which in turn corresponds to a $G$-torsor $E \arr T_{\ell}$. By applying the Weil restriction functor we obtain a morphism $\rR_{\ell/\kappa}E \arr \rR_{\ell/\kappa}(T_{\ell})$; since the Weil restriction commutes with fibered products, from the action of $G$ on $E$ we get an action $\rR_{\ell/\kappa}E \times_{\ell} \rR_{\ell/\kappa}G = \rR_{\ell/\kappa}(E \times G) \arr \rR_{\ell/\kappa}E$, which leaves the morphism $\rR_{\ell/\kappa}E \arr \rR_{\ell/\kappa}(T_{\ell})$ invariant. By pulling back along the unit morphism $T \arr  \rR_{\ell/\kappa}(T_{\ell})$ we obtain a morphism $F \arr T$, with an action of $\rR_{\ell/\kappa}G$ on $F$ leaving it invariant. It is easy to see that $F \arr T$ is a $\rR_{\ell/\kappa}G$-torsor: when $E = T_\ell \times G$ this follows from the fact that $\rR_{\ell/\kappa}$ preserves products. In general, if $T' \arr T$ is a faithfully flat morphism of affine schemes such that $T' \times_{T}E \simeq T' \times G$, we get a diagram
   \[
   \begin{tikzcd}
   T'\times\rR_{\ell/\kappa}G \rar\dar& F\dar\\
   T' \rar & T
   \end{tikzcd}
   \]
which is easily checked to be cartesian.

Thus we get a functor $\rR_{\ell/\kappa}(\cB_{\ell}G) \arr\cB_{\kappa}(\rR_{\ell/\kappa}G) $. Let us produce a functor in the opposite direction.

Let $T$ be an affine $\kappa$-scheme, $T \arr \cB_{\kappa}(\rR_{\ell/\kappa}G)$ a morphism, corresponding to a $\rR_{\ell/\kappa}G$-torsor $E \arr T$. Then the pullback $E_{\ell} \arr T_{\ell}$ is a $(\rR_{\ell/\kappa}G)_{\ell}$-torsor; with a change of group along the unit morphism $(\rR_{\ell/\kappa})_{\ell} \arr G$ we obtain a $G$-torsor $E \arr T_{\ell}$, which yields a morphism $T \arr \rR_{\ell/\kappa}(\cB_{\kappa}G)$.

We leave it to the reader to check that the two functors above are in fact quasi-inverses.
\end{proof}

\begin{proposition}\call{prop:Weil-restriction} Let $\cC$ be a stable class of groups over $\kappa$. Let $\ell/\kappa$ be a finite extension, and\/ $\Gamma$ a $\cC$-gerbe over $\ell$. Assume that either

\begin{enumerate1}

\itemref{1} $\cC$ is weakly very stable and $\ell/\kappa$ is separable, or

\itemref{2} $\cC$ is very stable class of groups over $\kappa$, and every extension of a group in $\cC$ by a product of copies of $\ga$ is still in $\cC$.

\end{enumerate1}

Then $\rR_{\ell/k}\Gamma$ is a $\cC$-gerbe.
\end{proposition}

\begin{proof}
Suppose that $\ell/\kappa$ is separable of degree~$d$, and let $\kappa'$ be a Galois closure of $\ell/\kappa$; call $\nu_{1}$, \dots,~$\nu_{d}$ the embeddings of $\ell$ in $\kappa'$. Denote by $\Gamma_{i}$ the fibered product $\Gamma_{\kappa'}$ induced by $\nu_{i}\colon \ell \arr \kappa'$. 

We have that $\ell\otimes_{\kappa}\kappa'$ is isomorphic to the $\kappa'$-algebra $(\kappa')^{d}$, where the $i\th$ projection $\ell\otimes_{\kappa}\kappa' \simeq (\kappa')^{d} \arr \kappa'$ corresponds to $\nu_{i}$. But $(\rR_{\ell/\kappa}\Gamma)_{\kappa'} \simeq \rR_{\ell\otimes_{\kappa}\kappa'/\kappa'}\Gamma_{\ell\otimes_{\kappa}\kappa'} \simeq \rR_{\ell\otimes_{\kappa}\kappa'/\kappa'}\coprod_{i=1}^d \Gamma_i \simeq \prod_{i=1}^d \Gamma_i  $, and each of the $\Gamma_{i}$ is a $\cC$-gerbe; since $\cC$ is weakly very stable, it follows that $\rR_{\ell/\kappa}\Gamma$ is a $\cC$-gerbe. This proves \refpart{prop:Weil-restriction}{1}.

For \refpart{prop:Weil-restriction}{2}, we may replace $\kappa$ with its separable closure in $\ell$, and assume that $\ell/\kappa$ is purely inseparable. Call $p$ the characteristic of $\kappa$; since the Weil restriction is functorial, we may assume $\ell = \kappa(\hspace*{-1pt}\radice{p}{a})$. 

Let $\kappa'$ a finite separable extension of $\kappa$, $\ell' \eqdef \kappa'\otimes_{\kappa}\ell$. Then $\ell' = \kappa'(\hspace*{-1pt}\radice{p}{a})$; it is enough to show that $(\rR_{\ell/\kappa}\Gamma)_{\kappa'} = \rR_{\ell'/\kappa'}(\Gamma_{\ell'})$ is a $\cC$-gerbe. Since every finite separable extension of $\ell$ is of the type above, we may assume that $\Gamma(\ell) \neq \emptyset$, so that $\Gamma = \cB_{\ell}G$ for some $\cC$-group $G$ over $\ell$. By Lemma~\ref{lem:weil-transfer-classifying}, we need to show that $\rR_{\ell/\kappa}G$ is a $\cC$-group. Since the class $\cC$ is by hypothesis very stable, it is enough to show that $(\rR_{\ell/\kappa}G)_{\ell}$ is a $\cC$-group.

For each non-negative integer $n$, denote by $A_{n}$ the $\ell$-algebra $\ell[x]/(x^{n+1})$. Then $\ell\otimes_{\kappa}\ell \simeq A_{p-1}$, and $(\rR_{\ell/\kappa}G)_{\ell}$ is the group scheme $\underhom_{\ell}(\spec A_{p-1}, G)$.

Let us show by induction on $n$ that the group scheme $G_{n} \eqdef\underhom_{\ell}(\spec A_{n}, G)$ is in $\cC$. In fact $G_{0} = G \in \cC$; the embedding $\spec A_{n-1} \subseteq \spec A_{n}$ induced a homomorphism $G_{n} \arr G_{n-1}$, whose image is a subgroup scheme of $G_{n-1}$, which is in $\cC$, by induction hypothesis. It is easy to see that its kernel is isomorphic to the Lie algebra of $G$, which a product of copies of $\ga$; hence $G_{n}$ is in $\cC$. This ends the proof.
\end{proof}

Let $X \arr \aff \kappa$ be a concentrated geometrically reduced fibered category, $\cC$ a well-founded class. Let $\kappa'/\kappa$ be a field extension. Consider the morphism $X_{\kappa'} \arr (\univ X)_{\kappa'}$ obtained by base change from the morphism $X \arr\univ X$; since $(\univ X)_{\kappa'}$ is a pro-$\cC$-gerbe, it will factor through $\univprime {X_{\kappa'}}$; so we obtain a morphism of $\kappa'$-gerbes $\univprime {X_{\kappa'}} \arr (\univ X)_{\kappa'}$.

\begin{theorem}\call{thm:base-change}
Assume one of the following hypotheses:

\begin{enumerate1}

\itemref{1} the class $\cC$ is weakly very stable, and $\kappa'/\kappa$ is an algebraic separable extension, or

\itemref{2} the class $\cC$ is very stable, every extension of a group in $\cC$ by a product of copies of $\ga$ is still in $\cC$, and $\kappa'/\kappa$ is an algebraic extension.

\end{enumerate1}

Then the map $\univprime {X_{\kappa'}} \arr (\univ X)_{\kappa'}$ is an equivalence.
\end{theorem}

\begin{proof}
The proof is virtually identical to the proof of \cite[Proposition~6.1]{borne-vistoli-fundamental-gerbe}, using Proposition~\ref{prop:Weil-restriction}.
\end{proof}

\section{The tannakian interpretations of the unipotent\\and virtually unipotent fundamental gerbes}\label{sec:tannakian}

Let $X$ be a concentrated fibered category over $\aff\kappa$ with $\H^{0}(X, \cO) = \kappa$, and $\cC$ a well-founded class. The gerbe $\univ{X}$ is tannakian, hence the category of representations $\rep\univ{X}$ is a tannakian category. For any $\cC$, one can ask if it is possible to give a direct description of $\rep\univ{X}$ in terms of $X$. Of course there is a pullback map $\rep\univ{X} \arr \vect_{X}$ into the category of locally free sheaves on $X$; however, this is in general not fully faithful. For example, when $\cC$ is the class of abelian group schemes it is immediate to see that the pullback functor is not fully faithful, for example, when there are maps between invertible sheaves on $X$ that are neither zero nor isomorphisms. And in fact we don't have a candidate for such a description of $\rep\univ{X}$ when $\cC$ is the class of abelian groups.

However when $\cC$ is the category of unipotent, or virtually unipotent, group schemes, the functor is fully faithful; and in fact in these cases there is a good description of $\rep\univ{X}$.

In what follows we will assume the following conditions on $X$.

\begin{conditions}\call{conditions-on-X}\hfil

\begin{enumerate1}

\itemref{0} $X$ is concentrated.

\itemref{1} $X$ is geometrically reduced.

\itemref{2} $\H^{0}(X, \cO) = \kappa$.

\itemref{3} $X$ is pseudo-proper, in the sense of \cite{borne-vistoli-fundamental-gerbe}.

\end{enumerate1}
\end{conditions}

Recall that $X$ is pseudo-proper if for any locally free sheaf $E$ on $X$, the $\kappa$-vector space $\H^{0}(X, E)$ is finite dimensional.

For example, an affine gerbe always satisfies these conditions.

We can't think of an example in which \refpart{conditions-on-X}{0}, \refpart{conditions-on-X}{1} and \refpart{conditions-on-X}{2} are satisfied but \refpart{conditions-on-X}{3} is not, but we have no doubt that this is for lack of trying.

Notice that if $X$ satisfies the conditions above and $\kappa'$ is a finite extension of $\kappa$, the fibered category $X_{\kappa'} \arr \aff{\kappa'}$ also satisfies them.

\begin{definition}\label{def:vect-u-vu}
Let $E$ be a locally free sheaf on $X$.

\begin{enumerate1}

\item We say that $E$ is \emph{unipotent} if it admits a filtration
   \[
   0 = E_{r+1} \subseteq E_{r} \subseteq E_{r-1} \subseteq \dots \subseteq E_{1} \subseteq E_{0} = E
   \]
in which all the quotients $E_{i}/E_{i+1}$ are free.

\item We say that $E$ is an \emph{extended essentially finite sheaf} if there is a filtration as above in which all the quotients $E_{i}/E_{i+1}$ are essentially finite.

\end{enumerate1}
\end{definition}

Unipotent bundles have been introduced by M.~Nori in \cite{nori-phd} under the name \emph{nilpotent bundles}. The second class of bundles has been introduced by S.~Otabe in \cite{otabe-extension}; he calls them \emph{semi-finite bundles}.

Both classes form a tannakian category, and have a natural interpretation in our language. 

\begin{definition}
A group scheme $G \arr \spec \ell$\/ is \emph{strongly virtually unipotent} if it  if it has a normal unipotent subgroup $H \subseteq G$ such that $G/H$ is finite.
\end{definition}

A strongly virtually unipotent group scheme is clearly virtually unipotent. If $G$ is smooth, then the converse holds; hence if $\cha \kappa = 0$, then the converse holds. This is not true in in positive characteristic, as the following examples show.

\begin{examples} Let us give two example of two group schemes in positive characteristic that are virtually unipotent, but not strongly virtually unipotent. The first one is abelian and defined over a non-perfect field, the second one is not abelian, but is defined over an arbitrary field of positive characteristic. The first example also tells us that the class of strongly virtually unipotent groups, which is weakly very stable, is not very stable.

\begin{enumerate1}\call{counterexamples}

\itemref{1} Let $k$ be a non-perfect field of characteristic $p > 0$, $\ell/k$ a purely inseparable extension of degree~$p$. Denote by $G$ the Weil transfer $\rR_{\ell/k}\mmu_{p}$. Then we claim that $G$ is virtually unipotent, but not strongly virtually unipotent.

We have $\ell\otimes_{k}\ell \simeq \ell[\epsilon] \eqdef \ell[x]/(x^{p})$; hence $G_{\ell}$ is the Weil transfer $\rR_{\ell[\epsilon]/\ell}\mmu_{p}$. This can be described as the group scheme
   \[
   H \eqdef \underhom_{\ell}(\spec \ell[\epsilon], \mmu_{p})
   \]
whose sections over an $\ell$-algebra $A$ are the homomorphisms of $A$-algebras
   \[
   A[t]/\bigl((t-1)^{p}\bigr) \arr A[\epsilon] = A[x]/(x^{p})\,.
   \]
These are uniquely determined by the image of $t$ in $A[\epsilon]$, which is an element $a_{0} + a_{1}\epsilon + \dots a_{p-1}\epsilon^{p-1}$ with $a_{0}^{p} = 1$; the product structure is given by the product in $A[\epsilon]$. There is a projection $H \arr \mmu_{p}$, defined by sending $a_{0} + a_{1}\epsilon + \dots a_{p-1}\epsilon^{p-1}$ into $a_{0}$, whose kernel is easily seen to be unipotent. Hence $G_{\ell}$ is an extension of $\mmu_{p}$ by unipotent group scheme, so $G$ is virtually unipotent, and has dimension $p-1$.

On the other hand, a homomorphism from a unipotent group scheme $U$ on $k$ to $G$ corresponds, by adjunction, to a homomorphism $U_{\ell}\arr \mmu_{p}$, which must be trivial; so $G$ does not contain any non-trivial unipotent subgroups, and is not strongly virtually unipotent.

\itemref{2} Here $k$ can be an arbitrary field of positive characteristic, $n$ an integer with $n > 1$. Call $\GL_{n}^{(1)}$ the Frobenius kernel in $\GL_{n}$, and $\rU_{n} \subseteq \GL_{n}$ the subgroup consisting of strictly upper triangular matrices. Let $G \eqdef \rU_{n} \ltimes \GL_{n}^{(1)}$, where the action of $\rU_{n}$ on $\GL_{n}^{(1)}$ is by conjugation. Clearly $G\ored = \rU_{n}$, so $G$ is virtually unipotent. 

Suppose that $H \subseteq G$ is a normal unipotent subgroup such that $G/H$ is finite. Then $\rU_{n} \subseteq H$, and $K \eqdef H\cap \GL_{n}^{(1)}$. Then $K$ is nontrivial, because the action of $\rU_{n}$ on $\GL_{n}^{(1)}$ is non-trivial, so $\rU_{n}$ is not normal in $G$. Since $K$ is unipotent we have that the invariant subspace $(k^{n})^{K}$ in $k^{n}$ is proper and non-trivial; but $K$ is normal in $\GL_{n}^{(1)}$, so $(k^{n})^{K}$ is $\GL_{n}^{(1)}$-invariant. But obviously $\GL_{n}^{(1)}$ is not contained in any proper parabolic subgroup of $\GL_{n}$, and this gives a contradiction, showing that $H$ can not exist, and that $G$ is not strongly virtually unipotent.

\end{enumerate1}

\end{examples}

In what follows we will denote by $\univu X$, $\univvu X$ and $\univsvu X$ the fundamental gerbe $\univ X$, when $\cC$ is, respectively, the class of unipotent, virtually unipotent, or strongly virtually unipotent groups.

\begin{theorem}\call{thm:tannakian-char-u-vu}\hfil
\begin{enumerate1}

\itemref{1} The pullback $\rep\univu{X} \arr \vect_{X}$ induces an equivalence of\/ $\rep\univu{X}$ with the full subcategory of $\vect{X}$ whose objects are unipotent locally free sheaves.

\itemref{2} The pullback $\rep\univsvu{X} \arr \vect_{X}$ induces an equivalence of\/ $\rep\univsvu{X}$ with the full subcategory of $\vect{X}$ whose objects are extended essentially finite sheaves.

\end{enumerate1}

\end{theorem}

So, if $\cha\kappa = 0$, this gives a tannakian interpretation of $\univvu X = \univsvu X$. 

Theorem~\ref{thm:tannakian-char-u-vu} is a particular case of the more general Theorem~\ref{thm:tannakian-characterization-1} in the next section.

There also a tannakian interpretation of $\univvu X$ in positive characteristic, at least with a weak additional assumption on $X$. Assume that $\cha\kappa = p > 0$.

If $U$ is a scheme over $\kappa$, denote by $\frob_{U}\colon U \arr U$ the absolute Frobenius map of $U$.

Denote by $\frob_{X}\colon X \arr X$ the functor sending an object $\xi \in X(T)$, where $T$ is an affine scheme over $\kappa$ to the pullback $\frob_{T}^{*}\xi$; this is a morphism of fibered categories over $\aff{\FF_{p}}$, not over $\aff \kappa$. Notice that this definition involves the choice of a cleavage for $X$; but the resulting functor is unique, up to a unique isomorphism.

If $U$ is an affine scheme, then $\frob_{\aff U}\colon \aff U \arr \aff U$ is the functor corresponding the morphism $\frob_{U}\colon U \arr U$; we will also denote it by $\frob_{U}$. If $U = \spec A$ we will also use the notation $\frob_{A}$.

Clearly, the diagram
   \[
   \begin{tikzcd}
   X \rar{\frob_{X}} \dar & X \dar\\
   \aff \kappa \rar{\frob_{\kappa}} & \aff \kappa
   \end{tikzcd}
   \]
is strictly commutative. If $F\colon X \arr Y$ is a morphism of fibered categories over $\aff\kappa$, then we have an obvious commutative diagram
   \[
   \begin{tikzcd}
   X \rar{\frob_{X}} \dar{F} & X \dar{F}\\
   Y \rar{\frob_{Y}} & Y\hsmash{\,.}
   \end{tikzcd}
   \]

\begin{definition}
A locally free sheaf $E$ on $X$ is \emph{virtually unipotent} if there exists a positive integer $n$, such that $(\frob_{X}^{n})^{*}E$ is an extended essentially finite sheaf.
\end{definition}

\begin{theorem}\label{thm:tannakian-characterization-3} Assume that $X$ has an fpqc cover $U \arr X$, where $U$ is a noetherian reduced scheme. Then the pullback\/ $\rep\univvu{X} \arr \vect_{X}$ induces an equivalence of the tannakian category $\rep\univvu{X}$ with the full subcategory of $\vect{X}$ consisting of virtually unipotent sheaves.
\end{theorem}

We do not know whether the (rather weak) condition on the existence of a cover $U \arr X$ as above is necessary for the conclusion to hold. It is certainly satisfied when $X$ is an affine gerbe, because any morphism from the spectrum of a field to an affine gerbe is an fpqc cover. 

Notice that from Theorem~\ref{thm:tannakian-characterization-3} and Corollary~\ref{cor:fully-faithful} we obtain the following.

\begin{corollary}\label{cor:fully-faithful-2}
If $\cC$ is a stable subclass of the class of virtually unipotent group schemes, then the pullback $\rep\univ X \arr \vect_{X}$ is fully faithful.
\end{corollary}

The following is due to Tonini and Zhang.

\begin{theorem}[\hbox{\cite[Corollary II]{tonini-zhang-divided}}]\label{thm:tonini-zhang}
Assume that $\cha \kappa > 0$, that $X$ is a pseudo-proper geometrically reduced algebraic stack of finite type over $\kappa$, and that\/ $\H^{1}(X, E)$ is a finite-dimensional vector space over $\kappa$ for all locally free sheaves on $X$. Then
   \[
   \univvu X = \univsvu X = \univfin X\,.
   \]
\end{theorem}

This is clearly false without the hypothesis on $\H^{1}$ (for example, take $X = \cB_{\kappa}\ga$).

\section{Unipotent saturations}\label{sec:general-tannakian}

Suppose that $\cV$ is a class of locally free sheaves on fibered categories satisfying Conditions~\ref{conditions-on-X}; for each such fibered category $X$ we denote by $\cV(X)$ the class of locally free sheaves on $X$ that are in $\cV$, and also the corresponding full subcategory of $\vect_{X}$.

\begin{definition}\call{def:tannakian-class}
Let $\cV$ be  a class of locally free sheaves on fibered categories satisfying Conditions~\ref{conditions-on-X}. We say that $\cV$ is a \emph{tannakian class} if it satisfies the following conditions (where $X$ and $Y$ are arbitrary fibered categories satisfying Conditions~\ref{conditions-on-X} and $\Gamma$ is an affine gerbe).

\begin{enumerate1}

\itemref{1} For each $X$, the subcategory $\cV(X) \subseteq \vect_{X}$ is a tannakian subcategory.

\itemref{2} If $f\colon Y \arr X$ is a morphism and $E$ in $\cV(X)$, then $f^{*}E$ is in $\cV(Y)$.

\itemref{3} Let $f\colon X \arr \Gamma$ a morphism such that the pullback $f^{*}\colon \rep\Gamma \arr \vect_{X}$ induces an equivalence between $\rep\Gamma$ and $\cV(X)$. Then $\cV(\Gamma) = \rep\Gamma$.

\end{enumerate1}
\end{definition}

Here by a \emph{tannakian subcategory} $\cV(X) \subseteq \vect_{X}$ we mean that it is a monoidal subcategory closed under isomorphisms and taking dual, that is tannakian with respect to the induced rigid monoidal structure, and that kernels and cokernels in $\cV(X)$ are also kernels and cokernels as homomorphism of sheaves of $\cO_{X}$-modules.

\begin{remark}
Condition \refpart{def:tannakian-class}{3} of the definition above may look strange; it has been introduced because it is essential for the proof of Lemma~\ref{lem:criterion}. We should point out that we don't examples of in which \refpart{def:tannakian-class}{1} and \refpart{def:tannakian-class}{2} hold, but \refpart{def:tannakian-class}{3} does not.
\end{remark}

There are many examples of tannakian classes, for example, the class of free locally free sheaves, and that of essentially finite locally free sheaves. Many more examples are provided by the following Proposition.

\begin{proposition}\label{prop:saturation-tannakian}
Let $\cC$ be a well-founded class of group schemes; assume that for each $X$ satisfying Conditions~\ref{conditions-on-X} the functor $\rep\univ X \arr \vect_{X}$ is fully faithful. Denote by $\cV(X)$ its essential image. Then $\cV$ is a tannakian class.
\end{proposition}

\begin{proof}
Straightforward.
\end{proof}

We will call this $\cV$ the \emph{tannakian realization} of the class $\cC$. Classes $\cC$ satisfying the condition of Proposition~\ref{prop:saturation-tannakian} will be called \emph{realizable}. Every stable class of finite group schemes is realizable, because of Corollary~\ref{cor:finite->realizable}.

The following gives a criterion to check that a tannakian class $\cV$ is the tannakian realization of a fundamental class $\cC$.

\begin{lemma}\call{lem:criterion}
Let $\cV$ be a tannakian class and $\cC$ a fundamental class of group schemes. Assume that for every affine gerbe $\Gamma$ over $\kappa$, a representation of\/ $\Gamma$ is in $\cV(\Gamma)$ if and only if it is the pullback of a representation of\/ $\univ{\Gamma}$.

Then $\cV$ is the tannakian realization of $\cC$.
\end{lemma}

\begin{proof}
If $\Gamma$ is an affine gerbe, we claim that $\Gamma$ is a pro-$\cC$-gerbe if and only if $\cV(\Gamma) = \rep \Gamma$. In fact, $\Gamma$ is a pro-$\cC$-gerbe if and only if $\Gamma \arr \univ\Gamma$ is an equivalence, that is, if and only if the pullback $\rep{\univ\Gamma} \arr \rep\Gamma$ is an equivalence. But $\rep{\univ\Gamma} \arr \rep\Gamma$ is fully faithful, by Proposition~\ref{prop:induced-locally-full}.

Let $\Pi$ be the affine gerbe corresponding to the tannakian category $\cV(X)$; this is a pro-$\cC$-gerbe, because of what we just showed. By Deligne's theorem \cite[Th\'eor\`eme 1.12]{deligne-tannakian} we obtain a map $X \arr \Pi$ such that the pullback $\rep\Pi \arr \vect_{X}$ induces an equivalence $\rep\Pi \simeq \cV(X)$. Hence $\cV(\Pi) = \rep\Pi$, because of part~\refpart{def:tannakian-class}{3} of Definition~\ref{def:tannakian-class}, so $\Pi$ is a pro-$\cC$-gerbe, because of the result above. Consider the factorization $X \arr \univ{X} \arr \Pi$, which induces a factorization $\rep\Pi \arr \rep\univ{X} \arr \vect{X}$. But $\rep\Pi$ and $\rep\univ{X}$ have the same essential image $\cV(X)$, $\rep\univ{X} \arr \vect_{X}$ is faithful, and $\rep\Pi \arr \vect_{X}$ is fully faithful, so $\rep\Pi \arr \rep\univ{X}$ is an equivalence. This concludes the proof.
\end{proof}

\begin{definition}
Let $\cV$ be a tannakian class. The \emph{unipotent saturation} $\overline{\cV}$ is defined as follows. Let $E$ be a locally free on $X$. We say that $E$ is in $\overline{\cV}$ if it admits a filtration
   \[
   0 = E_{r+1} \subseteq E_{r} \subseteq E_{r-1} \subseteq \dots \subseteq E_{1} \subseteq E_{0} = E
   \]
in which all the quotients $E_{i}/E_{i+1}$ are in $\cV$.
\end{definition}

Thus, the unipotent saturation of the class of free sheaves is the class of unipotent sheaves, while that of the class of essentially free sheaves is the class of extended essentially finite sheaves.

\begin{proposition}\label{prop:unipotent-saturation}
The unipotent saturation of a tannakian class is a tannakian class.
\end{proposition}

\begin{proof}
Let us check that the three conditions of Definition~\ref{def:tannakian-class} are satisfied. This is trivial for \refpart{def:tannakian-class}{2}.

Condition~\refpart{def:tannakian-class}{1} is easily proved by adapting the proof in \cite[Chapter~IV.1]{nori-phd} for unipotent locally free sheaves.

For condition~\refpart{def:tannakian-class}{3}, let $f\colon X \arr \Gamma$ be a morphism such that $f^{*}\colon \rep\Gamma \arr \vect_{X}$ induces an equivalence between $\rep\Gamma$ and $\overline{\cV}(X)$. Let $\Delta$ be the affine gerbe corresponding to the tannakian category $\cV(X)$. The embedding $\cV(X) \subseteq \overline{\cV}(X)$ induces a morphism $\phi\colon \Gamma \arr \Delta$, and a commutative diagram
   \[
   \begin{tikzcd}
   X \rar["f"]\ar[swap, dr, "g"]& \Gamma\dar{\phi}\\
   &\Delta
   \end{tikzcd}
   \]
such that $g^{*}\colon \rep\Delta \arr \vect_{X}$ induces an equivalence between $\rep\Delta$ and $\cV(X)$. Since $\cV$ is a tannakian class, we have $\cV(\Delta) = \rep\Delta$.

Let $E$ be a non-zero representation of $\Gamma$; we need to show that $E$ is in $\overline{\cV}(\Gamma)$. We proceed by induction on $\rk E$, and assume that all representations of $\Gamma$ of rank less than $\rk E$ are in $\overline{\cV}(\Gamma)$. If $f^{*}E$ is in $\cV(X)$, then it is the pullback of a representation $G_{0} \in \rep\Delta = \cV(\Delta)$, so $E \simeq \phi^{*}G_{0}$ is in $\cV(\Gamma)$, and we are done.

By hypothesis we have $f^{*}E \in \overline{\cV}(X)$; hence there exists a subsheaf $0 \neq F_{0} \subseteq f^{*}E$ with $F_{0} \in \cV(X)$. Let $E_{0}$ be a representation of $\Gamma$ with $f^{*}E_{0} \simeq F_{0}$; by the argument above, $E_{0} \in \cV\bigl(\Gamma\bigr)$. Since $f^{*}$ is fully faithful we get an embedding $E_{0} \subseteq E$. By induction hypothesis $E/E_{0}$ is in $\overline{\cV}(X)$, so $E \in \overline{\cV}(X)$.
\end{proof}

In general, the unipotent saturation of  the tannakian realization of a fundamental class of group schemes is not the tannakian realization of a fundamental class; for example, one can show that the unipotent saturation of the tannakian realization of the class of abelian group schemes is not a tannakian realization.

The main result of this section is that the tannakian realization of the unipotent saturation of a very stable fundamental class of finite group schemes is again a tannakian realization.

\begin{definition}
Let $\cD$ be a very stable class of finite group schemes. The \emph{unipotent saturation} $\overline{\cD}$ of $\cD$ is the class of affine algebraic groups $G$ over extensions $\ell$ of $\kappa$, with the property that there exists a normal unipotent subgroup scheme $H$ such that $G/H$ is in $\cD$.

\end{definition}

\begin{proposition}
The unipotent saturation of a very stable class of finite group schemes is well-founded.
\end{proposition}

\begin{proof}
Let $\cD$ be a weakly  very stable class of finite group schemes. Since $\overline{\cD}$ is a subclass of the class of all virtually finite group schemes, which is well-founded, it is enough to show that $\overline{\cD}$ is stable. This is straightforward.
\end{proof}

\begin{theorem}\label{thm:tannakian-characterization-1}
Let $\cD$ be a very stable class of finite group schemes, $\cV$ its tannakian realization. Then the tannakian realization of the unipotent saturation $\overline{\cD}$ is the unipotent saturation $\overline{\cV}$ of $\cV$.
\end{theorem}

When applied to the class $\cD$ consisting of trivial groups, and to the class of all finite group schemes, this immediately implies Theorem~\ref{thm:tannakian-char-u-vu}.

The proof of Theorem~\ref{thm:tannakian-characterization-1} will occupy the rest of this section. We use Proposition~\ref{prop:unipotent-saturation} and Lemma~\ref{lem:criterion}; we only have to check that the condition of Lemma~\ref{lem:criterion} is satisfied. This is the content of the following Proposition.

\begin{proposition}\label{prop:second}
If $\Gamma$ is an affine gerbe, a representation $E$ is in $\overline{\cV}(\Gamma)$ if and only if it is a pullback from $\Gamma^{\overline{\cD}}$.
\end{proposition}

For the proof of Proposition~\ref{prop:second} we need the following fact.

Let $\Gamma \arr \Delta$ be a morphism of gerbes over $\kappa$. Let us assume that this is locally full, or, equivalently, that $\Gamma$ is a gerbe over $\Delta$. Then we say that $\Gamma$ is \emph{unipotent over $\Delta$} if for any morphism $\spec \ell \arr \Delta$, where $\ell$ is a field, the fibered product $\spec \ell \times_{\Delta} \Gamma$ is unipotent.

It is easy to see that if $\spec \ell \times_{\Delta} \Gamma$ is unipotent for some morphism $\spec \ell \arr \Delta$, then it is unipotent for all such morphisms.

\begin{proposition}\label{prop:third}
Let $\Gamma$ be a gerbe of finite type over $\kappa$; then $\Gamma$ is a $\overline{\cD}$-gerbe if and only if\/ $\Gamma$ is unipotent over\/ $\Gamma^{\cD}$.
\end{proposition}

\begin{proof}[Proof of Proposition~\ref{prop:second}, assuming Proposition~\ref{prop:third}]

Let $\Gamma$ be an affine gerbe over $\kappa$.

Let us prove that every representation in $\overline{\cV}(\Gamma)$ comes from $\Gamma^{\overline{\cD}}$. This is obviously true for representations in $\cV(\Gamma)$, by definition; hence, it is enough to show that given an extension
   \[
   0 \arr E_{1} \arr E \arr E_{2} \arr0
   \]
in which both $E_{1}$ and $E_{2}$ come from $\Gamma^{\overline{\cD}}$, the representation $E$ also comes from $\Gamma^{\overline{\cD}}$. Let 
$\Gamma'$ be a $\overline{\cD}$-gerbe of finite type with a map $\Gamma \arr \Gamma'$ such that $E_{1}$ and $E_{2}$ come from representations $F_{1}$ and $F_{2}$ of $\Gamma'$.

Consider the fibered category $\Delta$ over $\aff\kappa$ defined as follows. Given an object $\xi \in \Gamma'(A)$, where $A$ is a $\kappa$-algebra, we denote by $(E_{1})_{\xi}$ and $(E_{2})_{\xi}$ the pullbacks to $\spec A$ obtained from the morphism $\spec A \arr \Gamma'$ corresponding to $\xi$. An object of $\Delta$ over a $\kappa$-algebra $A$ is an extension
   \[
   0 \arr (E_{1})_{\xi} \arr F \arr (E_{2})_{\xi} \arr 0
   \]
of sheaves of $\cO$-modules over $\spec A$.

The arrows in the fiber category $\Gamma(A)$ are given by homomorphisms of sheaves $F \arr F'$ of sheaves of $\cO$-modules over $\spec A$ fitting into a commutative diagram
   \[
   \begin{tikzcd}
   0\rar & (F_{1})_{\xi} \dar[equal] \rar & F \dar\rar
      & (F_{2})_{\xi} \rar \dar[equal] & 0\\
   0\rar & (F_{1})_{\xi} \rar & F' \rar & (F_{2})_{\xi} \rar  &  0  \,.
   \end{tikzcd}
   \]
   
It is immediate to check that $\Delta$ is a gerbe over $\Gamma'$. Furthermore, if $\ell$ is an extension of $\kappa$ and $\eta$ is an object of $\Delta(\ell)$ mapping to $\xi$ in $\Gamma'(\ell)$, the kernel of the natural homomorphism of group schemes $\underaut_{\ell}\eta \arr \underaut_{\ell}\xi$ is the vector space $\hom_{\ell}\bigl((E_{2})_{\xi}, (E_{1})_{\xi}\bigr)$; hence it is unipotent. Clearly, an extension of a $\overline{\cD}$-group with unipotent kernel is again in $\overline{\cD}$; hence $\underaut_{\ell}\eta$ is in $\overline{\cD}$, and $\Delta$ is a $\overline{\cD}$-gerbe.

The extension
   \[
   0 \arr E_{1} \arr E \arr E_{2} \arr 0
   \]
gives a lifting $\Gamma \arr \Delta$ of the given morphism $\Gamma \arr \Gamma'$; since $\Delta$ is a $\overline{\cD}$-gerbe, the morphism $X \arr \Delta$ factors through $\Gamma^{\overline{\cD}}$, and $E$ come from $\Gamma^{\overline{\cD}}$, as claimed.

In the other direction, we can replace $\Gamma$ with $\Gamma^{\overline{\cD}}$; so it enough to show that every representation of a $\overline{\cD}$-gerbe $\Gamma$ is in $\cV(\Gamma)$. This follows immediately from the following lemma, by an obvious induction on the rank.

\begin{lemma}\label{lem:nonzero-injective}
Let $\Gamma$ be a $\overline{\cD}$-gerbe over $\kappa$; denote by $\pi\colon \Gamma \arr \Gamma^{\cD}$ the projection. Let $E$ be a non-zero representation of\/ $\Gamma$. Then $\pi_{*}E \neq 0$, and the counit homomorphism $\pi^{*}\pi_{*}E \arr E$ is injective.
\end{lemma}

\begin{proof}
Let $\ell$ be an extension of $\kappa$ such that $\Gamma(\ell) \neq 0$. By Proposition~\ref{prop:third}, the pullback $\spec \ell \times_{\Gamma^{\cD}} \Gamma$ is of the form $\cB_{\ell}G$, where $G$ is a unipotent group scheme over $\ell$, and we have a cartesian diagram
   \[
   \begin{tikzcd}
   \cB_{\ell}G \rar{\psi}\dar{\rho}& \Gamma \dar{\pi}\\
   \spec \ell \rar{\phi} & \Gamma^{\cD}\,.
   \end{tikzcd}
   \]
Set $V \eqdef \psi^*E$. Since $\phi$ is faithfully flat, so is $\psi$. Furthermore, the formation of $\pi_{*}$ commutes with base change, so it is enough to show that $\psi^{*}\pi^{*}\pi_{*}E = \rho^{*}\rho_{*}V$ injects in $V$. But thinking of $V$ as a representation of $G$ we have $\rho^{*}\rho_{*}V = V^{G}$, and the statement is clear.
\end{proof}

\begin{remark}\label{rmk:counit-injective}
The argument in the proof of Lemma~\ref{lem:nonzero-injective} proves the following fact, that will be used later.

Let $\pi\colon \Gamma \arr \Delta$ be a locally full morphism of affine gerbes. If $E$ is a representation of $\Gamma$, the counit homomorphism $\pi^{*}\pi_{*}E \arr E$ is injective.
\end{remark}

This ends the proof of Proposition~\ref{prop:second}, assuming Proposition~\ref{prop:third}.
\end{proof}

\begin{proof}[Proof of Proposition~\ref{prop:third}]
Assume that $\Gamma$ is unipotent on $\Gamma^{\cD}$, and let $\xi$ be an object of $\Gamma(\ell)$ for some extension $\ell/\kappa$. Let $\eta$ be the image of $\xi$ in $\Gamma^{\cD}(\ell)$; we need to show that $\underaut_{\ell}\xi$ is 
a $\overline{\cD}$-group. Let $\eta$ be the image of $\xi$ in $\Gamma^{\cD}(\ell)$; the kernel of the natural surjective map $\underaut_{\ell}\xi \arr \underaut_{\ell}\eta$ is the automorphism of the object of $(\spec \ell \times_{\Gamma^{\cD}} \Gamma)(\ell)$ coming from $\xi$, so it is a unipotent group, while $\underaut_{\ell}\eta$ is in $\cD$.

For the other implication, let $\Gamma$ be a $\overline{\cD}$-gerbe. By Lemma~\ref{lem:finite-separable-point}, there exist a finite separable extension  $\ell/\kappa$ such that $\Gamma(\ell)\neq \emptyset$. Let $\xi$ be an object of $\Gamma(\ell)$; consider the fibered product $\Delta \eqdef \spec\ell \times_{\Gamma^{\cD}}\Gamma$. We need to show that $\Delta$ is a unipotent gerbe. By Theorem~\refall{thm:base-change}{1} we see that the formation of $\Delta$ commutes with base change along the morphism $\spec\ell \arr \spec\kappa$; hence, by base changing we may assume that $\Delta(\kappa) \neq \emptyset$. Let $\eta$ be the image of $\xi \in \Delta(\kappa)$ in $\Gamma(\kappa)$, and $\zeta$ its image in $\Gamma^{\cD}(\kappa)$. Set $H \eqdef \underaut_{\kappa}\xi$, $G \eqdef \underaut_{\kappa}\eta$ and $L \eqdef \underaut_{\kappa}\zeta$; then $H$ is a normal subgroup of $G$, and $L = G/H$. We need to show that $H$ is unipotent.

By hypothesis $G$ is in $\overline{\cD}(\kappa)$; hence there exists a normal unipotent subgroup $G' \subseteq G$ such that $G'' \eqdef G/G'$ is in $\cD(\kappa)$. The corresponding morphism $\cB_{\kappa}G \arr \cB_{\kappa}G''$ factors through $\cB_{\kappa}G \arr (\cB_{\kappa}G)^{\cD} = \cB_{\kappa}L$; this give a factorization $G \arr L \arr G''$, and shows that $H$ is included in $G'$. Since $G'$ is unipotent, so is $H$, as claimed.
\end{proof}

This ends the proof of Theorem~\ref{thm:tannakian-characterization-1}.

\section{The proof of Theorem \ref{thm:tannakian-characterization-3}}\label{sec:proof-3} 

By Lemma~\ref{lem:criterion}, to prove Theorem~\ref{thm:tannakian-characterization-3} it is enough to show the following two results.

\begin{proposition}\label{prop:tannakian-vu-1}
Assume that $X$ has an fpqc cover $U \arr X$, where $U$ is a noetherian reduced scheme. Then the  virtually unipotent locally free sheaves form a tannakian class.
\end{proposition}

\begin{proposition}\label{prop:tannakian-vu-2}
Let $\Gamma$ be an affine gerbe. Then a representation of\/ $\Gamma$ is virtually unipotent if and only if it is isomorphic to a pullback from $\rep\Gamma\vu$.
\end{proposition}

\begin{proof}[Proof of Proposition~\ref{prop:tannakian-vu-1}]

Let us check again that the three conditions of Definition~\ref{def:tannakian-class} are satisfied. This is clear for \refpart{def:tannakian-class}{2}.

For condition \refpart{def:tannakian-class}{1}, the only non-obvious thing to prove is that virtually unipotent locally free sheaves form an abelian subcategory; for this we need to show that the cokernel of a homomorphism of virtually unipotent locally free sheaves is again  locally free and virtually unipotent.

Let $\phi\colon F \arr G$ be a homomorphism of virtually unipotent sheaves, and let $Q$ be its cokernel. Choose an integer $n$ such that $(\frob_{X}^{n})^{*}F$ and $(\frob_{X}^{n})^{*}G$ are extended essentially finite sheaves on $X$; then $(\frob_{X}^{n})^{*}Q$ is the cokernel of the pullback homomorphism $(\frob_{X}^{n})^{*}\phi$, so it is an extended essentially finite sheaf on $X$, from Theorem~\refall{thm:tannakian-char-u-vu}{2}.

Choose an fpqc cover $\pi\colon U \arr X$, where $U$ is a noetherian reduced scheme. We have a commutative diagram
   \[
   \begin{tikzcd}
   U \rar{\frob_{U}^{n}} \dar{\pi}& U\dar{\pi}\\
   X \rar{\frob_{X}^{n}} & X\hsmash{\,.}
   \end{tikzcd}
   \]
Now, $\pi^{*}(\frob_{X}^{n})^{*} Q = (\frob_{U}^{n})^{*}\pi^{*}Q$ is locally free on $U$; since $\frob_{U}$ is a homeomorphism, we have that the function $U \arr \NN$ which sends $p \in U$ into $\dim_{k(p)}\bigl(\pi^{*}Q\otimes_{\cO_{U}}k(p)\bigr)$ is locally constant. We deduce that $\pi^{*}Q$ is locally free on $X$; hence $Q$ is a virtually unipotent locally free sheaf on $X$, and this concludes the proof.

Let us now check condition \refpart{def:tannakian-class}{3}. Let $f\colon X \arr \Gamma$ be a morphism such that $f^{*}\colon \rep\Gamma \arr \vect_{X}$ induces an equivalence between $\rep\Gamma$ and virtually  unipotent locally free sheaves on $X$. Let $E$ be a representation of $\Gamma$; we need to show that $E$ is virtually unipotent. We use induction on the rank $r$ of $E$, the result being clear for $r=0$. Assume $r>0$, and fix an integer $n$ such that $f^*(\frob_{\Gamma}^{n})^{*}E= (\frob_{X}^{n})^{*}f^* E$ is an unipotent sheaf on $X$. So there exists a positive integer $s$ and a monomorphism $\cO_X^{\oplus s} \to f^*(\frob_{\Gamma}^{n})^{*}E$. Since $f$ is fully faithful, this lifts into a monomorphism $\cO_\Gamma^{\oplus s} \to (\frob_{\Gamma}^{n})^{*}E$. Because the rank of $(\frob_{\Gamma}^{n})^{*}E / \cO_\Gamma^{\oplus s}$ is less than $r$, it is virtually unipotent. Since virtually unipotent sheaves are stable by extensions, we conclude that $(\frob_{\Gamma}^{n})^{*}E$ is virtually unipotent, which implies in turn that $E$ itself is virtually unipotent.
\end{proof}

\begin{proof}[Proof of Proposition~\ref{prop:tannakian-vu-2}]

The proof is somewhat long, so we split it into three steps.

\step{Step 1: reduction to the case that $\Gamma$ is of finite type}

Let us assume that the Proposition holds for affine gerbes of finite type; let $\Gamma$ be an arbitrary affine gerbe over $\kappa$, and let $V$ be a representation of $\Gamma$.

If $V$ comes from $\Gamma\vu$, we need to show that $V$ is virtually unipotent. Since the pullback of a virtually unipotent representation is virtually unipotent, we can replace $\Gamma$ with $\Gamma\vu$, and assume that $\Gamma$ is virtually unipotent.  Choose a locally full morphism $\Gamma \arr \Delta$, where $\Delta$ is a gerbe of finite type, and a representation $W$ of $\Delta$ whose pullback to $\Gamma$ is isomorphic to $V$. Then $\Delta$ is virtually unipotent, hence $W$ is virtually unipotent, and so $V$ is virtually unipotent.

Conversely, let $V$ be a virtually unipotent representation of $\Gamma$; we need to show that $V$ is a pullback from $\Gamma\vu$. Once again, choose a locally full morphism $\Gamma \arr \Delta$, where $\Delta$ is a gerbe of finite type, and a representation $W$ of $\Delta$ whose pullback to $\Gamma$ is isomorphic to $V$. It is enough to show that there exists a factorization $\Gamma \arr \Delta_{1} \arr \Delta$, where $\Delta_{1}$ is again an affine gerbe of finite type, such that the pullback $W_{1}$ of $W$ to $\Delta_{1}$ is virtually unipotent. In fact, this implies that $W_{1}$ is a pullback from $\Delta_{1}\vu$, and, since the composite $\Gamma \arr \Delta_{1}\vu$ factors through $\Gamma\vu$, the conclusion follow.

Assume that $V$ is an extended essentially finite sheaf. In this case we know from Proposition~\ref{prop:second} that $V$ comes from $\Gamma\svu$; choose a locally full morphism $\Gamma\svu \arr \Theta$ and a representation $Z$ of $\Theta$, whose pullback to $\Gamma$ is isomorphic to $V$. We can choose a locally full morphism $\Gamma \arr \Delta_{1}$, where $\Delta_{1}$ is an affine gerbe of finite type, such that both $\Gamma \arr \Delta$ and $\Gamma \arr \Theta$ factor through $\Delta_{1}$. If $Z_{1}$ denotes the pullback of $Z$ to $\Delta_{1}$, we have that $Z_{1}$ is an extended essentially finite sheaf. The pullback of $W_{1}$ and $Z_{1}$ to $\Gamma$ are both isomorphic to $V$; since the pullback $\rep\Delta_{1} \arr \rep\Gamma$ is fully faithful, it follows that $W_{1}$ and $Z_{1}$ are isomorphic, so $W_{1}$ is an extended essentially finite sheaf.

In the general case, choose a positive integer such that $(\frob_{\Gamma}^{m})^{*}V$ is an extended essentially finite sheaf; choose a factorization $\Gamma \arr \Delta_{1} \arr \Delta$ of the desired type, such that the pullback of $(\frob_{\Delta}^{m})^{*}W$ is an extended essentially finite sheaf; but this pullback is $(\frob_{\Delta_{1}}^{m})^{*}W_{1}$, and the conclusion follows.

\step{Step 2: reduction to the case that\/ $\Gamma(\kappa) \neq \emptyset$}

Since by Lemma~\ref{lem:finite-separable-point} there exists a finite Galois extension $\kappa'$ of $\kappa$ with $\Gamma(\kappa') \neq \emptyset$, the following lemma allows us to assume that $\Gamma(\kappa) \neq \emptyset$, so that $\Gamma = \cB_{\kappa}G$ for some affine group $G$ of finite type over $\kappa$. 

\begin{lemma}\call{lem:reduce-Galois}
Let $\kappa'$ be a finite Galois extension of $\kappa$. Let $V$ be a representation of $\Gamma$.

\begin{enumerate1}

\itemref{1} The representation $V$ is a pullback from $\Gamma\vu$ if and only if the pullback $V_{\kappa'}$ of $V$ to $\Gamma_{\kappa'}$ is a pullback from $(\Gamma_{\kappa'})\vu$.

\itemref{2} The representation $V$ is virtually unipotent if and only the pullback of $V$ to $\Gamma_{\kappa'}$ is virtually unipotent.

\end{enumerate1}
\end{lemma}

\begin{proof}
Call $G$ the Galois group of $\kappa'/\kappa$. By descent theory, we have an equivalence between $\rep\Gamma$ and the category of $G$-equivariant locally free sheaves on $\Gamma_{\kappa'}$.

For part \refpart{lem:reduce-Galois}{1}, assume that $V_{\kappa'}$ comes from a representation $W$ of $(\Gamma_{\kappa'})\vu$. Notice that by Theorem~\ref{thm:base-change} we have $(\Gamma_{\kappa'})\vu = (\Gamma\vu)_{\kappa'}$. For each $g \in G$ we denote by $g\colon \Gamma_{\kappa'} \arr \Gamma_{\kappa'}$ and $g\colon (\Gamma\vu)_{\kappa'} \arr (\Gamma\vu)_{\kappa'}$
the induced morphisms. The $G$-equivariant structure of $V_{\kappa'}$ gives a collections of isomorphisms $V_{\kappa'} \simeq g^{*}V_{\kappa'}$ of locally free sheaves over $\Gamma_{\kappa'}$; since the pullback $\rep(\Gamma\vu)_{\kappa'} \arr \rep\Gamma_{\kappa'}$ is fully faithful, these give isomorphisms $g^{*}\colon W \simeq g^{*}W$, that give $W$ a structure of $G$-equivariant representation on $(\Gamma\vu)_{\kappa'}$. Thus $W$ descends to a representation of $\Gamma\vu$, whose pullback to $\Gamma$ is isomorphic to $V$.

For part \refpart{lem:reduce-Galois}{2}, first of all notice that the argument above also applies to the $\Gamma\svu$ fundamental gerbe; since the representations coming from the $\Gamma\svu$ are the extended essentially finite representations, this shows that if $V_{\kappa'}$ is an extended essentially finite representation, then so is $V$.

Now fix $V$ a representation of $\Gamma$ so that $V_{\kappa'}$ is virtually unipotent. Since the pullback $(\frob_{\Gamma_{\kappa'}}^{m})^{*}V_{\kappa'}$ is isomorphic to $\bigl((\frob_{\Gamma}^{m})^{*}V\bigr)_{\kappa'}$, we may assume that $V_{\kappa'}$ is an extended essentially finite sheaf, and the conclusion follows. 
\end{proof}

\step{Step 3: the conclusion}

Let us show that every locally free sheaf on $\Gamma = \cB_{\kappa}G$ coming from $\Gamma\vu$ is virtually unipotent; since being virtually unipotent is a property that is stable under pullback, we may assume that $\Gamma = \Gamma\vu$, so that $G$ is virtually unipotent. Let $V$ be a representation of $G$; for $n \gg 0$ we have that the Frobenius morphism $\frob_{G}^{n}\colon G \arr G$ factors through $G\red$. Since $G\red$ is strongly virtually unipotent, we have that that pullback $(\frob_{\cB_{\kappa}G}^{n})^{*}V$ is strongly virtually unipotent, so $V$ is virtually unipotent.

Conversely, assume that $V$ is a virtually unipotent representation of $G$. We can replace $G$ with its image in $\GL(V)$, and assume that $V$ is faithful. Denote by $V^{(n)}$ the pullback of $V$ under the Frobenius map $\frob_{\cB_{\kappa}G}^{n}$; then for $n \gg 0$ the representation $V^{(n)}$ is strongly virtually unipotent. The kernel of $G \arr \GL(V^{(n)})$ is clearly finite; call $H$ the image of $G$ in $\GL(V^{(n)})$; then $V^{(n)}$ is strongly virtually unipotent as a representation of $H$. From Proposition~\ref{prop:second} it follows that $V^{(n)}$ is a pullback of a representation of a strongly virtually unipotent quotient of $H$; but $V^{(n)}$ is a faithful representation, so $H$ is strongly virtually unipotent. The following lemma allows us to conclude.

\begin{lemma}
Let $G$ be an affine algebraic group of finite type over $\kappa$. Assume that there exists a normal finite subgroup scheme $K \subseteq G$ such that $G/K$ is virtually unipotent. Then $G$ is virtually unipotent.
\end{lemma}

\begin{proof}
We can base change to the algebraic closure of $\kappa$, and assume that $\kappa$ is algebraically closed. Furthermore, if $H = G/K$, then $H\ored = G\ored/(G\ored \cap K)$; so we may assume that $G$, and therefore $H$, is smooth and connected. It follows that $H$ is unipotent; we will show that $G$ is also unipotent.

If $K$ is infinitesimal and we denote by $G^{[n]}$ the pullback under the isomorphism $\frob_{\kappa}^{n}\colon \spec\kappa \arr \spec\kappa$; then the relative Frobenius homomorphism $G \arr G^{[n]}$ factors through $G/K$ for $n \gg 0$. Hence $G^{[n]}$ is a quotient of $G/K$, so it is unipotent; it follows that $G$ is unipotent.

In the general case, by the previous case we can replace $G$ by $G/K^{0}$, and assume that $K$ is étale. Every smooth unipotent group scheme contains a normal subgroup scheme which is isomorphic to $\ga$; let $H' \subseteq H = G/K$ be such a subgroup scheme. Call $G'$ the connected component of the identity in the inverse image of $H'$ in $G$, and $K'$ the inverse image of $G'$ in $K$. Then $G/G'$ is an extension of $H/H'$ by $K/K'$; by induction on the dimension of $G$ we may assume that $G/G'$ is unipotent. Since every extension of unipotent groups is unipotent, it is enough to prove that $G'$ is unipotent. However, $G'$ is a smooth connected $1$-dimensional affine group scheme, so it is isomorphic to $\ga$ or to $\gm$. Since $G/K$ is isomorphic to $\ga$, $G$ cannot be $\gm$, so it is $\ga$. This concludes the proof.
\end{proof}

This ends the proof of Proposition~\ref{prop:tannakian-vu-2}, and of Theorem~\ref{thm:tannakian-characterization-3}.
\end{proof}

\section{Gerbes of multiplicative type and Picard stacks}

\subsection{Groups and gerbes of multiplicative type} Recall the following definitions.

Given an abelian group $A$, we consider the functor $\rD(A)\colon {\aff \kappa}\op \arr \catab$ that sends a $\kappa$-algebra $R$ into the group of homomorphisms of abelian groups $A \arr R^{\times}$. Then $\rD(A)$ is an affine group scheme; if $A$ is finitely generated, then $\rD(A)$ is of finite type. A group scheme over $\kappa$ is called \emph{diagonalizable} if it is isomorphic to some $\rD(A)$.

A group scheme $G$ over $\kappa$ is called \emph{of multiplicative type} if it satisfies one the following equivalent conditions.

\begin{enumerate1}

\item $G_{\kappa\sep}$ is diagonalizable.

\item $G_{\overline{\kappa}}$ is diagonalizable.

\item $G_{\ell}$ is diagonalizable for some extension $\ell$ of $\kappa$.

\item $G$ is commutative, and all representations of $G$ are semisimple.

\end{enumerate1}

The category of groups of multiplicative type is closed under taking subgroups, quotients and projective limits. Hence it coincides with the class of pro-$\cC$-groups, where $\cC$ is the class of multiplicative groups of finite type over $\kappa$.

Let us set up some notation.

If $R$ is a commutative ring, we will denote by $R\et$ the small étale site of $\spec R$, which we think of as the dual of the category of étale $R$-algebras.

If $R$ is a $\kappa$-algebra, we denote by $\cO^{\times}_{R\et}\colon R\et \arr \catab$ the sheaf $A \arr A^{\times}$. We shorten $\cO^{\times}_{\kappa\et}$ in $\cO^{\times}$.

We denote by $\gmet{R}\colon \kappa\et \arr \catab$ the pushforward of $\cO_{R\et}^\times$ to $\kappa\et$ via the morphism $R\et \arr \kappa\et$ induced by the homomorphism $\kappa \arr R$; this is the sheaf that sends an étale $\kappa$-algebra $A$ into the group $(A\otimes_{\kappa}R)^{\times}$. A homomorphism of $\kappa$-algebras $R \arr S$ induces a homomorphism of sheaves $\gmet{R} \arr \gmet{S}$.

Recall that the category of groups of multiplicative type over $\kappa$ is anti-equivalent to the category of sheaves of abelian groups over the small étale site $\kappa\et$ of $\spec k$; we will mostly think of $\kappa\et$ as the dual of the category of étale $\kappa$-algebras. With a group of multiplicative type $G$ we associate the sheaf of characters $\widehat{G}\colon \kappa\et \arr \catab$ defined as the functor sending each étale $\kappa$-algebra $A$ to the group of characters $\hom_{A}(G_{A}, \cO^{\times}_{A})$. 

In the other direction, given a sheaf $F$ on $\kappa\et$ we define a functor
   \[
   \rD(F)\colon \aff\kappa\op \arr \catab
   \]
by sending a $\kappa$-algebra $R$ into the group $\hom(F, \gmet{R})$.

Another way of stating this equivalence is the following. Let $\cG$ be the Galois group of $\kappa\sep$ over $\kappa$. The abelian group $M \eqdef \hom(G_{\kappa\sep}, \gmp{\kappa\sep})$  has a natural continuous action of $\cG$; sending $G$ into $M$ gives an equivalent between the opposite of the category of group schemes of multiplicative type over $\kappa$, with that of abelian groups with a continuous action of $\cG$.

We will call a pro-$\cC$-gerbe, where $\cC$ is the class of groups of multiplicative type and finite type, a \emph{gerbe of multiplicative type}. Equivalently, a gerbe of multiplicative type is an affine gerbe banded by a group of multiplicative type, not necessarily of finite type.

The aim of this section is give a description of gerbes of multiplicative type that is very similar in spirit to that of groups of multiplicative type given above.

\subsection{Picard stacks} In this subsection we recall some known facts about Picard stacks.

By a \emph{Picard stack} $P$ over $\kappa$ we will mean, as in \cite[Exposé~XVIII]{sga4-3}, a stack in strictly commutative 
monoidal groupoids $P \arr \kappa\et$, whose operation $P \times_{\kappa\et} P \arr P$ is denoted by $(\xi, \eta) \arrto \xi\otimes\eta$, such that every object of $P$ is invertible. We will denote by $\epsilon\colon \kappa\et \arr P$ be the section corresponding to the identity; if $A \in \kappa\et$, image of $A$ in $P(A)$ will be denoted by $\epsilon_{A}$. If $\phi\colon A \arr B$ is a homomorphism of étale $\kappa$-algebras, the image of the corresponding arrow $\spec B \arr \spec A$ in $\kappa\et$ will be denoted by $\epsilon_{\phi}\colon \epsilon_{B} \arr \epsilon_{A}$.

We denote by $\pi_{0}(P)$ the sheafification of the presheaf on $\kappa\et$ sending $A$ into the abelian group of isomorphism classes in $P(A)$. By $\pi_{1}(P)$ we denote the sheaf of automorphisms of the identity section $\epsilon\colon \kappa\et \arr P$. Because of the monoidal structure of $P$, the sheaf of groups $\pi_{1}(P)$ is abelian. The inverse image of the identity in $\pi_{0}(P)$ in $P$ is equivalent to the classifying stack $\cB_{\kappa\et}\pi_{1}(P)$.

Furthermore, if $\xi \in P(A)$ we have a homomorphism of groups
   \[
	   \pi_{1}(P)(A) =\aut_{A}(\epsilon_{A}) \arr \aut_{A}(\epsilon_{A}\otimes\xi)
   \]
   sending $\alpha$ into $\alpha\otimes\id_{\xi}$; by using the given isomorphism $\epsilon_{A}\otimes \xi \simeq \xi$ we obtain an group homomorphism $\pi_{1}(P)(A) \arr \aut_{A}\xi$, which is in fact an isomorphism. This gives an equivalence of Picard stacks between the inertia $I_{P}$ and the product $\pi_{1}(P)\times P$. 

   The rigidification $P\thickslash \pi_{1}(P)$ (in the sense of \cite[Appendix A]{dan-olsson-vistoli1}) is a sheaf, isomorphic to $\pi_{0}(P)$. Thus, $P$ can be thought of as a central extension of the sheaf $\pi_{0}(P)$ by the classifying stack $\cB\pi_{1}(P)$.

A homomorphism $P \arr Q$ of Picard stacks is an equivalence if and only if the induced homomorphisms of sheaves $\pi_{1}(P) \arr \pi_{1}(Q)$ and $\pi_{0}(P) \arr \pi_{0}(Q)$ are isomorphisms.

Picard stacks over $\kappa\et$ form a strict $2$-category, whose $1$-arrows are symmetric monoidal base-preserving functors; we call these \emph{homomorphisms of Picard stacks}.

In \cite[Exposé~XVIII]{sga4-3}, Deligne also showed how to describe Picard stacks in terms of complexes of sheaves of abelian groups: a Picard stack is a quotient $[L^{0}/L^{-1}]$ where $L^{\smallbullet}$ is a complex of sheaves on $\kappa\et$ concentrated in degrees $-1$ and $0$, and $L^{-1}$ is an injective sheaf. Such complexes form a $2$-category, the $2$-arrows being given by homotopies. Sending $L^{\smallbullet}$ into $[L^{0}/L^{-1}]$ gives an equivalence between the $2$-category of complexes of this type, and the $2$-category of Picard stacks.

If we fix two complexes of sheaves $G_{0}$ and $G_{1}$, consider the $2$-category of $2$-extensions of $G_{0}$ by $G_{1}$. An object of this $2$-category is a Picard stack $P$, with a fixed isomorphism $\pi_{0}(P) \simeq G_{0}$ and $\pi_{1}(P) \simeq G_{1}$. As a corollary of Deligne's result, we have that equivalence classes of such extensions are parametrized by $\ext^{2}_{\kappa\et}(G_{0}, G_{1})$.

\subsection{$\cO^{\times}$-stacks}

A simple example of a Picard stack is the classifying stack
   \[
   \cB_{\kappa\et}\cO^{\times} \arr \kappa\et\,;
   \]
it is the Picard stack in which every $\cB_{\kappa\et}\cO^{\times}(A)$ has a unique object, whose automorphism group is $A^{\times}$. Notice that it is a stack in the étale topology, because every étale $\kappa$-algebra has trivial Picard group.

We will often shorten $\cB_{\kappa\et}\cO^{\times}$ in $\cB\cO^{\times}$.

\begin{definition}
An $\cO^{\times}$-stack over $\kappa$ is a Picard stack $P$ over $\kappa$, together with a homomorphism of sheaves of groups $\rho_{P}\colon \cO^{\times} \arr \pi_{1}(P)$.

An $\cO^{\times}$-stack is \emph{rigid} if $\rho_{P}\colon \cO^{\times}_{\kappa\et} \arr \pi_{1}(P)$ is an isomorphism.
\end{definition}

Equivalently, an $\cO^{\times}$-stack is a Picard stack with a homomorphism of Picard stacks $\cB_{\kappa\et}\cO^{\times} \arr P$.

There is an obvious $2$-category of $\cO^{\times}$-stacks over $\kappa$: a homomorphism of $\cO^{\times}$-stacks $\phi\colon P \arr Q$ is a homomorphism of Picard stacks, such that the composite $\cO^{\times} \xarr{\rho_{P}} \pi_{1}(P) \xarr{\pi_{1}(\phi)} \pi_{1}(Q)$ equals $\rho_Q$. 

A homomorphism $P \arr Q$ of rigid $\cO^{\times}$-stacks is an equivalence if and only if the induced homomorphism  $\pi_{0}(P) \arr \pi_{0}(Q)$ is an isomorphism.

Every rigid $\cO^{\times}$-stack is a quotient stack $[L^{0}/L^{-1}]$, , with a fixed isomorphism $\ker(L^{-1}\to L^{0}) \simeq \cO^{\times}$. We will call complexes of this type \emph{$\cO^{\times}$-complexes}. They form a $2$-category; the $2$-arrows are given by homotopies. We have a functor from the $2$-category of $\cO^{\times}$-complexes into $\cO^{\times}$-stacks that sends $L^{\smallbullet}$ to $[L^{0}/L^{-1}]$; restricting this functor to $\cO^{\times}$-complexes $L^{\smallbullet}$ with $L^{-1}$ injective gives an equivalence of $2$-categories.

This implies that given a sheaf $F$ on $\kappa\et$, equivalence classes of rigid $\cO^{\times}$-stacks $P$ with a fixed isomorphism $\pi_{0}(P) \simeq F$ are classified by $\ext^{2}(F, \cO^{\times})$, where the $\ext^{2}$ is taken in the category of sheaves on $\kappa\et$.

\subsection{From gerbes to rigid $\cO^{\times}$-stacks} Let $X \arr \aff \kappa$ be a fibered category. If $A$ is an étale $\kappa$ algebra, the composite $X_{A} \arr \aff A \arr \aff \kappa$ makes $X_{A}$ into a category fibered over $\aff\kappa$.

The  Picard stack $\underpic_{X/\kappa} = \underpic_{X} \arr \kappa\et$ is the fibered category corresponding to the pseudo-functor on $\kappa\et$ that sends $A$ into the groupoid $\underpic(X_{A})$ of invertible sheaves on $X_{A}$. It can be conveniently defined as the stack of morphism $X_{A} \arr \cB_{\kappa}\gm$.

The stack $\underpic_{X}$ has a canonical structure of $\cO^{\times}$-stack. Furthermore, if $X$ is concentrated and $\H^{0}(X, \cO) = \kappa$, then $\underpic_{X}$ is rigid.

In particular, if $\Gamma$ is a gerbe of multiplicative type over $\kappa$, its Picard stack $\underpic_{\Gamma} \arr \kappa\et$ is a rigid $\cO^{\times}$-stack.
Notice that $\underpic_{\aff\kappa}$ is equivalent to $\cB\cO^{\times}$.

A morphism of fibered categories $X \arr Y$ over $\aff\kappa$ yields a homomorphism of $\cO^{\times}$-stacks $\underpic_{Y} \arr \underpic_{X}$; this defines a strict $2$-functor from the $2$-category of fibered categories over $\aff \kappa$ to the $2$-category of $\cO^{\times}$-stacks over $\kappa$.

If $R$ is a $\kappa$ algebra, we set $\underpic_{R} \eqdef \underpic_{\spec R}$; then $\pi_{1}(\underpic_{R}) = \gmet{R}$.

\begin{proposition}\label{prop:underpic-split}
Let $G \arr \spec\kappa$ be a group of multiplicative type. Then
   \[
   \underpic_{\cB_{\kappa} G} \simeq \cB_{\kappa\et}\cO^{\times} \times \widehat{G}\,.
   \]
\end{proposition}

\begin{proof}
There is an equivalence between invertible sheaves on $\cB_{A}G_{A} = (\cB_{\kappa}G)_{A}$ and morphisms $\cB_{A}G_{A} \arr \cB_{A}\gmp{A}$ over $\aff A$. If $A$ is an étale $\kappa$-algebra, by \cite[III Remarque~1.6.7]{giraud} the category of such morphisms is equivalent to the category of global object of the stack $\cB_{\kappa}\gm(A) \times \widehat{G}(A)  \simeq (\cB_{\kappa\et}\cO^{\times} \times \widehat{G})(A)$, and this gives the desired equivalence.
\end{proof}

For general gerbe we have the following.

\begin{proposition}\label{prop:band-picard}
Let $\Gamma$ be a gerbe of multiplicative type; call $G$ its band. Then $\pi_{0}(\underpic_{\Gamma})$ is canonically isomorphic to $\widehat{G}$.

\end{proposition}

\begin{proof}
Suppose that $\Gamma(\kappa) \neq \emptyset$; then the choice of an object $\xi$ of $\Gamma(\kappa)$ gives an isomorphism between $\pi_{0}(\underpic_{\Gamma})$
and $\widehat{G}$, where $G=\underaut_{\kappa}\xi$. This is independent of the choice of $\xi$; it is also functorial under finite separable extensions of the base field.

In the general case, for each separable extension $\kappa'/\kappa$ such that $G(\kappa') \neq \emptyset$ we obtain an isomorphism between the band of $\Gamma_{\kappa'}$, which is the restriction of the band of $\Gamma$, and $\widehat{G}_{\kappa'}$. These isomorphism are canonical, and descend to the desired isomorphism between the band of $\Gamma$ and $\widehat{G}$.
\end{proof}

Consider a cofiltered system $\{\Gamma_{i}\}_{i \in I}$ of affine gerbes over $\kappa$, and set $\Gamma \eqdef \projlim_{i}\Gamma_{i}$. For each arrow $j \arr i$ in $I$ we have a morphism $\Gamma_{j} \arr \Gamma_{i}$, hence a homomorphism of $\cO^{\times}$-stacks $\underpic_{\Gamma_{i}} \arr \underpic_{\Gamma_{j}}$. This gives a strict $2$-functor from $I\op$ to Picard stacks. The projection $\projlim \Gamma_{i} \arr \Gamma_{i}$ induces a homomorphism of Picard stacks $\underpic_{\Gamma_{i}} \arr \underpic_{\projlim_{i}\Gamma_{i}}$, and consequently a homomorphism $\indlim_{i}\underpic_{\Gamma_{i}} \arr \underpic_{\projlim_{i}\Gamma_{i}}$.

\begin{proposition}\label{prop:inductive-limit}
The homomorphism $\indlim_{i}\underpic_{\Gamma_{i}} \arr \underpic_{\projlim_{i}\Gamma_{i}}$ above is an equivalence.
\end{proposition}

\begin{proof}
This follows from \cite[Proposition~3.8]{borne-vistoli-fundamental-gerbe}.
\end{proof}

\subsection{From rigid $\cO^{\times}$-stacks to gerbes} We can also go from $\cO^{\times}$-stacks $P$ to fibered categories over $\aff\kappa$ in the following way. Let $P$ be an $\cO^{\times}$-stack. Let us define a fibered category $\underger_{P} \arr \aff\kappa$ as follows. If $R$ is a $\kappa$-algebra, an object of $\underger_{P}(R)$ is a homomorphism of $\cO^{\times}$-stacks $P \arr \underpic_{R}$. A morphism $\spec S \arr \spec R$ in $\aff\kappa$ induces a pullback functor $\underpic_{R} \arr \underpic_{S}$. Composing with this gives the function $R \mapsto \underger_{P}(R)$ a structure of pseudo-functor; we define $\underger_{P}$ to be the associated fibered category. (The notation here is a little improper, as $\underger_{P}$ is not necessarily a gerbe, if $P$ is not rigid.)

This construction gives a strict $2$-functor from the $2$-category of $\cO^{\times}$-stacks to that of stacks in groupoids on $\aff\kappa$.

\begin{proposition}\label{prop:is-a-gerbe}
	Assume that $P$ is a rigid $\cO^{\times}$-stack; then the fibered category $\underger_{P}$ is an affine gerbe. Its band is the group scheme of multiplicative type $\rD(\pi_{0}(P))$ corresponding to $\pi_{0}(P)$.
\end{proposition}

\begin{proof} We will divide the proof into several steps.

\step{Step 1: $\underger_{P}$ is an fpqc stack} This is a straightforward exercise in descent theory.

\step{Step 2: formation of $\underger_{P}$ commutes with separable algebraic extensions of the base field} Suppose that $\ell$ is a separable algebraic extension of $\kappa$. The obvious functor $\kappa\et \arr \ell\et$ sending $A$ into $A\otimes_{\kappa}\ell$ induces a morphism of sites $\phi\colon \ell\et \arr \kappa\et$, which gives a pullback functor $\phi^{-1}$ from sheaves, or stacks, over $\kappa\et$ to sheaves, or stacks, over $\ell\et$. Since $\phi^{-1}\cO_{\kappa\et} = \cO_{\ell\et}$, we have that $P_{\ell} \eqdef \spec\ell \times_{\spec\kappa}P$ is an $\cO^{\times}$-stack over $\ell$. It is immediate to conclude that $\underger_{P_{\ell}} = \spec\ell\times_{\spec\kappa}\underger_{P}$.

\step{Step 3: describing $\underger_{P}$ in terms of complexes of sheaves}


Let $R$ be a $\kappa$-algebra; call $\phi\colon R\et \arr \kappa\et$ the morphism of sites induced by the homomorphism $\kappa \arr R$. Fix an embedding $\cO^{\times}_{R\et}\subseteq I^{-1}$ into an injective sheaf, set $I^{0} \eqdef I^{-1}/\cO^{\times}_{R\et}$, and extend the projection $I^{-1}\arr I^{0}$ into a complex $I^{\smallbullet}$. Then, under Deligne's correspondence the Picard stack $\underpic_{R/\kappa}$ corresponds to the complex $\phi_{*}I^{\smallbullet}$.

Now, represent $P$ as a quotient $[L^{0}/L^{-1}]$, were $L^{\smallbullet}$ is an $\cO^{\times}$-complex. From the adjunction between pushforward and pullback, we obtain the that $\underger_{P}(R)$ is the equivalent to the category that has the following description.

\begin{enumeratea}

\item An object is an homomorphism $\xi\colon \phi^{-1}L^{\smallbullet} \arr I^{\smallbullet}$ of sheaves on $\kappa\et$, such that $H^{-1}(\xi)\colon \phi^{-1}\cO^{\times}_{\kappa\et} \arr \cO^{\times}_{R\et}$ is the canonical homomorphism.

\item An arrow $\xi \arr \eta$ is a homomorphism of graded sheaves $a\colon \phi^{-1}L^{\smallbullet} \arr I^{\smallbullet}[-1]$ such that $\eta - \xi = \delta_{I^{\smallbullet}}\circ a + a\circ \delta_{\phi^{-1}L^{\smallbullet}}$.

\end{enumeratea}

\step{Step 4: identifying the sheaf of automorphisms of an object of $\underger_{P}(R)$} From the description above it is immediate to check the automorphism group of an object $\xi$ is the group 
   \begin{align*}
   \underhom_{R\et}\bigl(\coker(\phi^{-1}L^{-1}\rightarrow \phi^{-1}L^{0}), \cO^{\times}_{R}\bigr)
   &= \underhom_{R\et}\bigl(\phi^{-1}\pi_{0}(P), \cO^{\times}_{R}\bigr)\\
   &= \underhom_{\kappa\et}\bigl(\pi_{0}(P), \gmet{R}\bigr) \\
   &= \rD(\pi_{0}(P)) 
   \end{align*}
as claimed.

\step{Step 5: $\underger_{P}$ is non-empty} Let us show that $\underger_{P}(\kappa\sep) \neq \emptyset$, where $\kappa\sep$ is a separable closure of $\kappa$. By step~2 we can base-change to $\kappa\sep$, and assume that $\kappa$ is separably closed. In this case $\ext^{2}\bigl(\pi_{0}(P), \cO^{\times}\bigr) = 0$, so that $P \simeq \cB\cO^{\times}\times \pi_{0}(P)$, and composing the projection $P \arr \cB\cO^{\times}$ with the natural homomorphism of Picard stacks $\cB\cO^{\times} \arr \underpic_{\spec \kappa}$ we obtain an object of $\underger_{P}(\kappa)$.

\step{Step 6: two objects of\/ $\underger_{P}$ are fpqc locally isomorphic} Once again, we may assume that $\kappa$ is separably closed, so that $P$ is of the form $\cB\cO^{\times} \times \pi_{0}(P)$. Let $R$ be a $\kappa$-algebra, $\xi$ and $\eta$ two objects of $\underger_{P}(R)$.

%

\begin{lemma}\call{lem:acyclic-cover}
For any $\kappa$-algebra $R$, there exists a faithfully flat extension $R \subseteq S$, such that
\begin{enumeratea}

\itemref{1} every faithfully flat étale ring homomorphism $S \arr T$ has a retraction $T \arr S$, and

\itemref{2} the abelian group $S^{\times}$ is divisible.

\end{enumeratea}
\end{lemma}

\begin{proof}
For part~\refpart{lem:acyclic-cover}{1} we will switch to the language of affine schemes, as this seems to make the proof somewhat more intuitive. Set $X \eqdef \spec R$. First of all, let us show there exists an faithfully flat map $Y \arr X$ where $Y$ is an affine scheme with the property that every étale surjective map $Z \arr Y$, where $Z$ is an affine scheme, has a section $Y \arr Z$.

The construction is straightforward, and mirrors the standard construction of the strict henselization at a point. The existence of $S$ also follows from \cite[Corollary 2.2.14]{bhatt-scholze-pro-etale}, for the convenience of the reader, we give a direct proof. For each $p\in X$ choose a separable closure $\kappa(p)\sep$, and set $X_{0} \eqdef \sqcup_{p \in X}\spec\bigl(\kappa(p)\sep\bigr)$. Let us define a category  $I$ whose objects $(U, f, a)$ are commutative diagrams 
   \[
   \begin{tikzcd}
   &U\dar{f}\\
   X_{0}\rar\ar[ur, "a"]&X
   \end{tikzcd}
   \]
   in which the bottom morphism is the obvious map, and the vertical one is an étale map with $U$ affine. A map from $(U, f, a)$ to $(V, g, b)$ is, naturally, a morphism $\phi\colon U \arr V$ with $f = g\phi$ and $b = \phi a$. This category has fibered products and a terminal object, so it is cofiltered. Define $Y = \projlim_{I}U$; the natural map $Y \arr X$ is faithfully flat. If $Z \arr Y$ is an étale surjective map with $Z$ affine, then $Z$ is a pullback $W\times_{U} Y$ of an étale surjective map $\phi\colon W \arr U$ for some object $(U, f, a)$ of $I$ (\cite[Tag 00U2]{stacks-project}). The map $a\colon X_{0} \arr U$ lifts to a map $b\colon X_{0} \arr W$; so $\phi$ is an arrow $(W, f\phi , b) \arr (U, f, a)$. Hence $Y \arr U$ lifts to $Y \arr W$, and this gives a section $Y \arr Z$. 

Now, set $S \eqdef \cO(Y)$; this $S$ has property \refpart{lem:acyclic-cover}{1}. If $\cha\kappa = 0$ and $f \in S^{\times}$, the extension $S \subseteq S[x]/(x^{n} - f)$ is étale, hence it has a section $S[x]/(x^{n} - f) \arr S$, which means exactly that $f$ has an $n\th$ root, so $S$ also has property~\refpart{lem:acyclic-cover}{2}.

Assume that $\cha \kappa = p > 0$; in this case the argument above only works when $p$ does not divide $n$, and we need to extend $S$ further, so that it has $p^{n}$-th roots of all the elements of $S^{\times}$ for all $n$. Let us start by replacing $R$ with $S$, and assume that $R$ has property~\refpart{lem:acyclic-cover}{1}. 

Consider the set $J$ of finite subsets $S \subseteq R^{\times}$, ordered by inclusion. For each $S \in J$ define an $R$-algebra $R[S]$ by adding an indeterminate $x_{f}$ for each $f \in S$, and dividing by the ideal generated by the polynomials $x_{f}^{p} - f$ for all $f \in S$. The algebra $R[S]$ is clearly faithfully flat, and the induced morphism $\spec R[S] \arr \spec R$ is a universal homeomorphism.

If $S \subseteq T$ there is an induced homomorphism $R[S] \arr R[T]$ defined by sending $[x_{f}] \in R[S]$ into $[x_{f}] \in R[T]$ for each $f \in S$. Set
   \[
   R_{1} \eqdef \indlim_{S\in J}R[S]\;;
   \]
from the construction it is clear that every $f \in R^{\times}$ has a $p\th$ root in $R_{1}$. Also, $R_{1}$ is faithfully flat over $R$, and the induced morphism $\spec R_{1} \arr \spec R$ is a universal homeomorphism.

Let us iterate this construction: for each $n > 0$ define $R_{n} \eqdef (R_{n-1})_{1}$. We get a sequence of faithfully flat extensions $R \subseteq R_{1} \subseteq \dots \subseteq R_{n} \subseteq \dots $ such that every induced morphism $\spec R \arr \spec R_{n-1}$ is a universal homeomorphism, and every element of $R_{n-1}^{\times}$ has  $p\th$ root in $R_{n}$. Set $S \eqdef \indlim_{n} R_{n}$; it follows that $S^{\times}$ is $p$-divisible.

Furthermore, $S$ is faithfully flat over $R$, and $\spec S \arr \spec R$ is a universal homeomorphism; from this it follows that pullback induces an equivalence between affine étale surjective maps to $\spec R$ and affine étale surjective maps to $\spec S$. Since $R$ has property~\refpart{lem:acyclic-cover}{1}, it follows that $S$ has it too. But this implies that $S^{\times}$ is $n$-divisible for all $n$ not divisible by $p$.
\end{proof}

So, we may assume that $R = S$. Consider the morphism of sites $\phi\colon R\et \arr \kappa\et$ given by the homomorphism $\kappa \arr R$. Condition~\refpart{lem:acyclic-cover}{1} of Lemma~\ref{lem:acyclic-cover} implies that sending a sheaf $F$ into its group of global sections is an exact functor; in other words, the pushforward $\phi_{*}$ from sheaves on $R\et$ to $\kappa\et$ is exact.

Choose an injective sheaf $I^{-1}$ containing $\cO^{\times}_{R\et}$, and set $I^{0} \eqdef I^{-1}/\cO^{\times}_{R\et}$, as in Step~3; hence $\phi_{*}I^{\smallbullet}$ is an injective resolution of $R^{\times}$. Now, $P$ is represented by the complex $\cO^{\times} \xarr{0} \pi_{0}(P)$, where $\pi_{0}(P)$ is in degree $0$; hence the difference $\eta - \xi$ gives a homomorphism of sheaves from  $\cO^{\times} \xarr{0} \pi_{0}(P)$ to $\phi_{*}I^{\smallbullet}$, which is $0$ in degree~$-1$; we need to show that this is homotopic to $0$, which is clear, because the sequence
   \[
   0 \arr R^{\times} \arr \phi_{*}I^{-1} \arr \phi_{*}I^{0} \arr 0
   \]
is split exact, because, by condition~\refpart{lem:acyclic-cover}{2} of Lemma~\ref{lem:acyclic-cover}, the abelian group $R^{\times}$ is injective.

This ends the proof of Proposition~\ref{prop:is-a-gerbe}.
\end{proof}

Thus we get a functor from the $2$-category of rigid $\cO^{\times}$-stacks to the category of gerbes of multiplicative type sending $P$ into $\underger_{P}$.

As a particular case of Proposition~\ref{prop:is-a-gerbe}, we obtain the following description of the gerbe associated with a split $\cO^{\times}$ stack $\cB\cO^{\times} \times A$.

\begin{corollary}\label{cor:gerbe-split-Picard}
Let $A$ be a sheaf of abelian groups in $\kappa\et$. Then $\underger_{\cB\cO^{\times} \times A}$ is canonically equivalent to $\cB_{\kappa}\rD(A)$.
\end{corollary}

\begin{proof}
The projection $\cB\cO^{\times} \times A \arr \cB\cO^{\times} = \underpic_{\kappa}$ gives an object $\xi \in \underger_{\cB\cO^{\times} \times A}(\kappa)$, hence an equivalence $\underger_{\cB\cO^{\times} \times A} \simeq \cB \underaut_{\kappa}\xi$. But according to Proposition~\ref{prop:is-a-gerbe} $\underaut_{\kappa}\xi$ is canonically isomorphic to $\rD(A)$.
\end{proof}

\subsection*{The main result}

If $X$ is a concentrated fibered category, it is easy to construct a functor $\Phi_{X}\colon X \arr \underger_{\underpic_{X}}$. Suppose that we are given an object $\xi$ of $X(R)$, corresponding to a morphism $\spec R \arr X$, which in turn gives a symmetric monoidal functor $\underpic_{X} \arr \underpic_{R}$, which is, by definition, an object of $\underger_{\underpic_{X}}(R)$. This gives a function on the objects, that extends easily to a base-preserving functor $\Phi_{X}\colon X \arr \underger_{\underpic_{X}}$.

A morphism of fibered categories $X \arr Y$ gives morphisms of fibered categories $\underpic_{Y} \arr \underpic_{X}$ and $\underger_{\underpic_{X}} \arr \underger_{\underpic_{Y}}$; it is straightforward to show that $\Phi$ is a natural transformation of $2$-functors from the identity to $\underger_{\underpic_{?}}$.

\begin{proposition}\label{prop:equivalence-1}
If $\Gamma$ is a gerbe of multiplicative type, the morphism $\Phi_{\Gamma}\colon \Gamma\arr \underger_{\underpic_{\Gamma}} $ is an equivalence.
\end{proposition}

\begin{proof}
We split the proof into three steps.

\step{Step 1} Assume that $\Gamma(\kappa) \neq \emptyset$, so that $\Gamma\simeq \cB_\kappa G$ for some group scheme of multiplicative type $G$ over $\kappa$. By Proposition~\ref{prop:underpic-split} we have a canonical equivalence $\underpic_{\cB_\kappa} \simeq \cB\cO^{\times} \times \widehat{G}$; by Corollary~\ref{cor:gerbe-split-Picard} we have $\underger_{\underpic_{\cB_\kappa}} \simeq \cB_{\kappa}\rD(\widehat{G})$. The composite $\cB_\kappa G \arr  \cB_{\kappa}\rD(\widehat{G})$ is easily seen to be isomorphic to the functor induced by the canonical map $G \arr \rD(\widehat{G})$, which is an isomorphism; this proves the result.

\step{Step 2} Assume that $\Gamma$ is of finite type over $\kappa$. Then, by Proposition~\ref{prop:gerbe-finite-type} we have that $\Gamma$ is a smooth algebraic stack over $\kappa$, hence there exists a finite separable extension $\kappa'/\kappa$ such that $\Gamma(\kappa') \neq \emptyset$. 

Since $\Gamma$ and $\underger_{\underpic_{\Gamma}}$ are fpqc stacks, is enough to show that the morphism $(\Phi_{X})_{\kappa'}\colon \Gamma_{\kappa'} \arr (\underger_{\underpic_{\Gamma}})_{\kappa'}$ of gerbes over $\arr\kappa'$ is an equivalence. But it is obvious that $\underpic_{\Gamma_{\kappa'}}$ is the restriction of $\underpic_{\Gamma}$ to $\kappa'\et$; hence $(\underger_{\underpic_{\Gamma}})_{\kappa'} =\underger_{\underpic_{\Gamma_{\kappa'}}} $. So, since $\Gamma_{\kappa'}(\kappa') \neq \emptyset$, the result follows from the first step.

\step{Step 3} In the general case, write $\Gamma$ as a projective limit $\projlim_{i}\Gamma_{i}$ of affine gerbes of finite type over~$\kappa$. Then by Proposition~\ref{prop:inductive-limit} we have that $\underpic_{\Gamma} \simeq \indlim_{i}\underpic_{\Gamma_{i}}$, and it is easy to see that $\underger_{\indlim_{i}\underpic_{\Gamma_{i}}} \simeq \projlim_{i}\underger_{\underpic_{\Gamma_{i}}}$. This completes the proof.
\end{proof}

If $P$ is an $\cO^{\times}$-stack, let us construct a homomorphism of $\cO^{\times}$-stacks $\Psi_{P}\colon P \arr \underpic_{\underger_{P}}$. Suppose that $\eta$ is an object of $P(A)$ for some étale $\kappa$-algebra $A$; we need to define an object $\Psi_{P}(\eta) \in \underpic_{\underger_{P}}(A)$, that is, a morphism of fibered categories $\Psi_{P}(\eta)\colon (\underger_{P})_{A} \arr \cB_{\kappa}\gm$.

Assume that we have an object $\phi$ of $(\underger_{P})_{A}(R)$; this consists of an $A$-algebra structure on the $\kappa$-algebra $R$, and an object of $\phi$ of $\underger_{P}(R)$, that is, a symmetric monoidal functor $\phi\colon P \arr \underpic_{R}$.  By applying $\phi$ to $\eta$ we obtain an element $\phi(\eta)$ of $\underpic_{R}(A) = \underpic(A\otimes R)$; we define $\Psi_{P}(\eta)(\phi) \eqdef \phi(\eta)$ as the image in $\underpic(R)$ of $\phi(\eta)$, via the functor $\underpic(A\otimes R) \arr \underpic(R)$ induced by the product homomorphism $A\otimes R \arr R$ (the one that gives the $A$-algebra structure on $R$). This defines $\Psi_{P}(\eta)$ at the level of objects. This is easily extended to a morphism of fibered categories.

We leave to the reader the straightforward, but dull, task of defining $\Psi_{P}$ as a symmetric monoidal functor, and to check that this makes $\Psi$ into a natural transformation of $2$-functors from the identity to $\underpic_{\underger_{?}}$.

\begin{proposition}\label{prop:equivalence-2}
If $P$ is rigid, the homomorphism $\Psi_{P}\colon P \arr \underpic_{\underger_{P}}$ is an equivalence.
\end{proposition}

\begin{proof}
We need to check that the homomorphism $\pi_{0}(P) \arr \pi_{0}(\underpic_{\underger_{P}})$ induced by $\Psi_{P}$ is an isomorphism. It follows from Propositions \ref{prop:band-picard} and \ref{prop:is-a-gerbe} that $\pi_{0}(\underpic_{\underger_{P}})$ is canically isomorphic to $\widehat{\rD(\pi_{0}(P))}$; and one checks that the homomorphism $\pi_{0}(P) \arr \widehat{\rD(\pi_{0}(P))}$ induced by $\Psi_{P}$ is the biduality map, which is an isomorphism.
\end{proof}

Propositions \ref{prop:equivalence-1} and \ref{prop:equivalence-2} imply the following.

\begin{theorem}\label{thm:main-equivalence}
Sending a gerbe $\Gamma$ of multiplicative type into $\underpic_{\Gamma}$, and a rigid $\cO^{\times}$-stack $P$ into $\underger_{P}$, gives an equivalence between the $2$-category of gerbes of multiplicative type and the opposite of the $2$-category of rigid $\cO^{\times}$-stacks.
\end{theorem}

Here is the main result of this section.

\begin{theorem}\label{thm:main-multiplicative}
Let $X \arr \aff\kappa$ be a concentrated fibered category such that\/ $\H^{0}(X, \cO) = \kappa$. Then the morphism $\Phi_{X}\colon X \arr \underger_{\underpic_{X}}$ makes $\underger_{\underpic_{X}}$ into the fundamental gerbe of $X$ for the class of group schemes of multiplicative type.
\end{theorem}

\begin{proof}
Let $\Gamma \arr \aff\kappa$ be a gerbe of multiplicative type; we need to show that the functor
   \[
   \hom(\underger_{\underpic_{X}}, \Gamma) \arr \hom(X, \Gamma)
   \]
induced by composition with $\Phi_{X}\colon X \arr \underger_{\underpic_{X}}$ is an equivalence.

Set $P \eqdef \underpic_{\Gamma}$; by Theorem~\ref{thm:main-equivalence} we have $\Gamma\simeq \underger_{P}$. There is a  $2$-commutative diagram of functors
   \[
   \begin{tikzcd}
   \hom(\underger_{\underpic_{X}}, \Gamma) \ar[rr]\ar[rd] && \hom(X, \Gamma)\ar[ld]\\
   &\hom(P, \underpic_{X})\op\,.
   \end{tikzcd}
   \]
The leftmost arrow is an equivalence, by Theorem~\ref{thm:main-equivalence}; hence to prove the theorem it is enough to show that the rightmost arrow is an equivalence. This holds for any concentrated stack, as the following lemma states; this completes the proof of the theorem.

\begin{lemma}
Let $X$ be a concentrated stack and $P$ a rigid $\cO^{\times}$-stack. Then the natural functor
   \[
   \hom(X, \underger_{P}) \arr \hom(P, \underpic_{X})
   \]
is an equivalence
\end{lemma}

\begin{proof}
If $X = \spec R$, this follows from the definition of $\underger_{P}$, and Yoneda's Lemma. The general case follows easily from this.

%
\end{proof}\noqed
\end{proof}

\bibliographystyle{amsplain}
\bibliography{fundamental.bib}

\end{document}